\documentclass[a4paper]{article}
\usepackage{amssymb}
\usepackage{amsmath}
\usepackage{mathrsfs}
\usepackage{amsfonts}
\usepackage{color}
\usepackage{vmargin}
\usepackage{amsthm}

\long\def\symbolfootnote[#1]#2{\begingroup%
\def\thefootnote{\fnsymbol{footnote}}\footnote[#1]{#2}\endgroup}

\setmarginsrb{20mm}{20mm}{20mm}{20mm}{10mm}{10mm}{10mm}{10mm}

\newtheoremstyle{remark}
  {}{}{}{}{\bfseries}{.}{.5em}{{\thmname{#1 }}{\thmnumber{#2}}{\thmnote{ (#3)}}}

\usepackage[english]{babel}
\input epsf

\newtheorem{theo}{Theorem}[section]
\newtheorem{lem}{Lemma}[section]

\newtheorem{prop}{Proposition}[section]
\newtheorem{example}{Example}
\setmarginsrb{20mm}{20mm}{20mm}{20mm}{10mm}{10mm}{10mm}{10mm}

\title{\bf Schauder theory in variable H\"{o}lder spaces}
\author{ Piotr Micha{\l} Bies \& Przemys{\l}aw G\'orka \\
\it\small{Department of Mathematics and Information Sciences,}\\
\it\small{Warsaw University of Technology,}\\
\it\small{Ul. Koszykowa 75, 00-662 Warsaw, Poland.}\\
{\tt pgorka@mini.pw.edu.pl}}

\begin{document}

\maketitle

\begin{abstract}
We study elliptic equations on bounded domain of Euclidean spaces in the variable H\"{o}lder spaces. Interior a priori Schauder estimates are given as well as global ones. 
Moreover, the existence and the uniqueness of solutions to the Dirichlet boundary value problem is proved.
\end{abstract}
\bigskip\noindent
{\bf Keywords}: Variable H\"{o}lder spaces; Schauder estimates; Kellog theorem; Variable exponent spaces; Dirichlet boundary value problem.

\bigskip\noindent
{\bf 2010 Mathematics Subject Classification:} 35J25; 26A16; 47B38.

\date
\section{Introduction}

Let $\Omega$ be a bounded domain in $\mathbb{R}^n$ with sufficiently smooth boundary $\partial \Omega$.  
The purpose of this paper is to investigate the linear elliptic operator $L$ defined as follows
\begin{equation*}
Lu(x):=a^{ij}(x)D_{ij}u(x)+b^i(x)D_iu(x)+c(x)u(x), \quad x\in \Omega,
\end{equation*}
where the coefficients $a^{ij}, b^i, c$ are in the variable H\"{o}lder space $C^{\alpha(\cdot)}(\bar{\Omega})$ 
(notation will be explained in the next section) and where the Einstein summation convention is used. 
We are interested in the study of the following boundary value problem
 \begin{align*}
\left\{\begin{array}{l}
L u=f \quad \text{in} \, \Omega,\\
u=\phi \quad \text{ on} \, \partial\Omega,
\end{array}\right.
\end{align*}
where $f$ and $\phi$ are elements of the variable H\"{o}lder spaces. Assuming some mild condition on the exponent $\alpha$, i.e. 
the so-called log-H\"{o}lder regularity, we show the interior a priori Schauder estimates as well as the global ones for the above 
elliptic boundary value problem. 
Moreover, under some regularity assumptions on the boundary $\partial \Omega$, we prove that the above Dirichlet problem posses a 
unique solution in the variable H\"{o}lder space $C^{2,\alpha(\cdot)}(\bar{\Omega})$. 
  
Linear elliptic operators involving variable coefficients in different function spaces 
have been extensively studied in the mathematical literature. Especially, the literature devoted to the elliptic 
problems in the classical H\"{o}lder spaces is very vast (see \cite{alm, Azzam, baz, ber, Bolley, can, campa, dou, fazio, krylov, krylovPriola, Mazya, naj, pic, sha, vak, wid, zvy} and 
references therein). Our results are natural extensions of the Schauder estimates as well as the Kellog theorems for variable H\"{o}lder spaces. 

The paper is divided into six sections. In Preliminaries we recall and introduce some basic notations and briefly present the theory of variable 
exponent spaces. 
Then, in Section~3 we prove a priori estimates for the Laplace operator. 
 Section~4 is devoted to studying the fully elliptic equation. In Theorem~\ref{sches} we show the interior Schauder estimates and subsequently 
in Theorem \ref{schesend} the global Schauder estimate is shown. In Section 5 we discuss the existence and uniqueness of solution to the Dirichlet 
problem. In the last section we present the interpolation theorems in the variable H\"{o}lder spaces.

\section{Preliminaries}

Let $\Omega\subset \mathbb{R}^n$ be an open set. We let $B(x, r)=\{y \in \mathbb{R}^{n} : |x - y|<r\}$ denotes a
ball centered at a point $x$ with radius $r>0$. We now turn to a presentation of the theory of the variable H\"{o}lder spaces \cite{AlSamko}, \cite{alm}. 
For background on variable exponent function spaces we refer to the monographs by Cruz-Uribe-Fiorenza \cite{Cruz} and Diening--Harjulehto--H\"ast\"o--R\r u\v zi\v cka~\cite{DHHR}. 

A function $\alpha \colon \Omega\to (0,1]$ is called a variable exponent or H\"older exponent, and we denote 
\[
\alpha^+_A:=\sup_{x\in A} \alpha(x),\quad \alpha^-_A:=\inf_{x\in A} \alpha(x),
\quad \alpha^+:=\alpha_\Omega^+ \quad\text{and}\quad \alpha^-:=\alpha_\Omega^-
\]
for $A\subset\Omega$.  If $A=\Omega$ or if the underlying domain is fixed, we will often skip the index.

By $C^{k}(\Omega)$ we denote the set of functions defined on $\Omega$, such that all derivatives up to the order $k$ are bounded and continuous,  
and by $C^k(\bar{\Omega})$ we denote the set of $u \in C^{k}(\Omega)$ such that all derivatives up to the order $k$  can be extended 
continuously to $\bar{\Omega}$. Let $u$ be a function defined on the set $\Omega$ and $\alpha\colon \Omega\to (0,1]$, then we introduce the 
following seminorm
$$
[u]_{\alpha(\cdot),\Omega}=\sup_{\substack{x,y\in \Omega\\ x\neq y}}\frac{|u(x)-u(y)|}{|x-y|^{\alpha(x)}}.
$$
Next, if $u\in C^k(\Omega)$ then we define
$$
[u]_{k,0,\Omega}=\sup_{\substack{x \in \Omega\\|\beta|=k}}|D^{\beta}u(x)|,
$$
\begin{equation}\label{semihold}
[u]_{k,\alpha(\cdot),\Omega}=\sup_{\substack{x,y\in \Omega\\ x\neq y\\|\beta|=k}}\frac{|D^{\beta}u(x)-D^{\beta}u(y)|}{|x-y|^{\alpha(x)}}.
\end{equation}
If $u\in C^k(\Omega)$, then we define the norm as follows
\begin{equation} \label{nor}
 \|u\|_{C^k(\Omega)}=|u|_{k,\Omega}=\sum\limits_{j=0}^k[u]_{j,0,\Omega}.
\end{equation}
Now, we define the spaces $C^{k,\alpha(\cdot)}(\Omega)$ and $C^{k,\alpha(\cdot)}(\bar{\Omega})$, respectively. 
The space $C^{k,\alpha(\cdot)}(\Omega)$ consists of those functions $u$ from $C^k(\Omega)$ such that for each compact subset $D \subset \Omega$, the 
quantity $|u|_{k,D} + [u]_{k,\alpha(\cdot),D}$ is finite. While, the space $C^{k,\alpha(\cdot)}(\bar{\Omega})$ 
consists of these functions from $C^k(\bar{\Omega})$ such that the quantities (\ref{semihold}) and (\ref{nor}) are finite. 
For $u\in C^{k,\alpha(\cdot)}(\bar{\Omega})$ we define the following norm
\begin{equation}\label{normhold}
\|u\|_{C^{k,\alpha(\cdot)}(\bar{\Omega})}=|u|_{k,\Omega}+[u]_{k, \alpha(\cdot),\Omega}.
\end{equation}
The space $C^{k,\alpha(\cdot)}(\bar{\Omega})$ equipped with the above norm is a Banach space.

Subsequently, we introduce the modified norms and seminorms respectively on the space
$C^{k, \alpha(\cdot)}(\bar{\Omega})$. Let $d=\hbox{diam}({\Omega})$, then
$$
\|u\|'_{C^k(\bar{\Omega})}=|u|'_{k,\Omega}=\sum\limits_{j=0}^kd^j|D^ju|_{0,\Omega},
$$
$$
\|u\|'_{C^{k,\alpha(\cdot)}(\bar{\Omega})}=|u|'_{k,\alpha(\cdot),\Omega}=|u|'_{k,\Omega}+\sup_{\substack{x,y\in \Omega\\x\neq y\\|\beta|=k}} d^{k+\alpha(x)}\frac{|D^{\beta}u(x)-D^{\beta}u(y)|}{|x-y|^{\alpha(x)}}.
$$
Let $x\in\Omega$, in many proofs we will frequently use the following notation
$$
[f]_{\alpha(\cdot), x}=\sup_{\substack{y\in \Omega \\y\neq x}}\frac{|f(x)-f(y)|}{|x-y|^{\alpha(x)}}.
$$
In order to get our results, it is necessary to put some restrictions on the variable exponents. 
Now, we introduce the class of sufficiently regular variable exponents. Namely, we shall say that $\alpha : \Omega \rightarrow \mathbb{R}$ is 
log-H\"older continuous if there 
exists a positive constant $M$ such that for each $x, y \in \Omega$ the following inequality holds 
$$
\left|\ln|x-y|\right|\left|\alpha(x)-\alpha(y)\right|\leq M.
$$ 
Next, we introduce the class of log-H\"older continuous exponents
\begin{equation*}
\mathcal{A}^{\log}(\Omega)=\{\alpha\colon \ 0<\alpha^-\leq\alpha^+<1,\ \alpha\textrm{ is log-H\"older continuous in}\ \Omega\}.
\end{equation*}
From now on, we assume that $\alpha$ is log-H\"older continuous exponent. Additionally, for a given $\alpha \in \mathcal{A}^{\log}(\Omega)$, the smallest constant 
for which $\alpha$ is log-H\"older continuous is denoted by $c_{\log}(\alpha)$. Let us mention that in the theory of variable exponent spaces, the 
log-H\"older continuity is a commonly used assumption on the variable exponents (see \cite{Cruz,DHHR}). 

Finally, we recall the notion of the Newtonian potential. 
Let $\Gamma$ be a fundamental solution of the Laplace equation in the Euclidean space. Then, for integrable function $f$ on a domain $D$, the Newtonian potential 
$w$ of the function $f$ is defined as follows
\begin{eqnarray*}
  w(x)=\int_D \Gamma (x-y) f(y) dy.
\end{eqnarray*}

\section{A priori estimates for Poisson's equation}
In this section we study the Poisson's equation
\begin{eqnarray*}
  \Delta u = f
\end{eqnarray*}
in the variable H\"{o}lder spaces. In particular, we show a priori estimates for the above equation. For this purpose, the detailed 
analysis of the Newtonian potential in the variable H\"{o}lder spaces will be needed. Let us start our journey with the following lemma.  
\begin{lem}\label{lem1}
Let $\Omega\subset\mathbb{R}^n$ be an open and bounded set and $w$ be the Newtonian potential of $f$ on $B_2=B(x_0, 2R)\subset \Omega$, where $f\in C^{\alpha(\cdot)}(\bar{B}_2)$. 
Then, $w\in C^{2,\alpha(\cdot)}(\bar{B}_1)$ and the following inequality is satisfied
\begin{equation}
|D^2w|'_{0,\alpha(\cdot), B_1}\leq C|f|'_{0,\alpha(\cdot),B_2},
\end{equation}
where $B_1=B(x_0, R)$ and $C=C(\textup{diam}(\Omega), n, \alpha^-, \alpha^+, c_{\log}(\alpha))$.
\end{lem}
\begin{proof}
For $x\in B_1$, we have
$$
D_{ij}w(x)=\int_{B_2}D_{ij}\Gamma(x-y)(f(y)-f(x))dy-f(x)\int_{\partial B_2}D_i\Gamma(x-y)\nu_j(y)dS(y),
$$
what yields
\begin{eqnarray*}
&&|D_{ij}w(x)|\leq\frac{|f(x)|}{n\omega_n}R^{1-n}\int_{\partial B_2}dS(y)+\frac{[f]_{\alpha(\cdot), x}}{\omega_n}\int_{B_2}|x-y|^{\alpha(x)-n}dy\\
&\leq&2^{n-1}|f(x)|+\frac{[f]_{\alpha(\cdot), x}}{\omega_n}\int_{B(x, 3R)}|x-y|^{\alpha(x)-n}dy\leq 2^{n-1}|f(x)|+[f]_{\alpha(\cdot), x}(3R)^{\alpha(x)}\frac{n}{\alpha(x)}\\
&\leq& C_1\left(|f(x)|+[f]_{\alpha(\cdot), x}\frac{(3R)^{\alpha(x)}}{\alpha(x)}\right).
\end{eqnarray*}
Thus, we get 
\begin{eqnarray*}
|D_{ij}w|_{0, B_1}\leq C|f|'_{0,\alpha(\cdot),B_2}.
\end{eqnarray*}

Now, let us fix $\bar{x}\in B_1$ and let us denote $\delta=|x-\bar{x}|$, $\zeta=\frac{1}{2}(x+\bar{x})$. Then, we can write 
$$
D_{ij}w(\bar{x})-D_{ij}w(x)=f(x)I_1+(f(x)-f(\bar{x}))I_2+I_3+I_4+(f(x)-f(\bar{x}))I_5+I_6,
$$
where $I_1,\ I_2,\ I_3,\ I_4,\ I_5,\ I_6$ are given by
\begin{eqnarray*}
    &I_1&=\int_{\partial B_2}(D_i\Gamma(x-y)-D_i\Gamma(\bar{x}-y))\nu_j(y)dS(y),\\
    &I_2&=\int_{\partial B_2}D_i\Gamma(\bar{x}-y)\nu_j(y)dS(y),\\
    &I_3&=\int_{B(\zeta,\delta)\cap B_2}D_{ij}\Gamma(x-y)(f(x)-f(y))dy,\\
    &I_4&=\int_{B(\zeta, \delta)\cap B_2}D_{ij}\Gamma(\bar{x}-y)(f(y)-f(\bar{x}))dy,\\
    &I_5&=\int_{B_2 \setminus B(\zeta,\delta)}D_{ij}\Gamma(x-y)dy,\\
    &I_6&=\int_{B_2 \setminus B(\zeta, \delta)}(D_{ij}\Gamma(x-y)-D_{ij}\Gamma(\bar{x}-y))(f(\bar{x})-f(y))dy.
\end{eqnarray*}
By the Mean Value Theorem we have the inequality
\begin{eqnarray*}
&&|D_i\Gamma(x-y)-D_i\Gamma(\bar{x}-y)|=\left|\int_0^1\frac{d}{dt}D_i\Gamma(tx+(1-t)\bar{x} -y)dt\right|\\
&\leq&\sup_{t\in[0,\ 1]}|x-\bar{x}||DD_i\Gamma(tx+(1-t)\bar{x}-y)|.
\end{eqnarray*}
Hence, we can estimate $I_1$ as follows
\begin{eqnarray*}
&|I_1|&\leq\int_{B_2}\sup_{t\in[0,\ 1]}|x-\bar{x}|\,|DD_i\Gamma(tx+(1-t)\bar{x}-y)|dy\\&\leq&\frac{n}{\omega_n}|x-\bar{x}|\int_{B_2}\sup_{t\in[0,\ 1]}|(tx+(1-t)\bar{x}-y|^{-n}dy\\&\leq&\frac{n}{\omega_n}|x-\bar{x}|R^{-n}\int_{B_2}dy=\frac{n}{\omega_n}|x-\bar{x}|R^{-n}(2R)^{n-1}\omega_nn=\frac{n^22^{n-1}|x-\bar{x}|}{R}\\
&\leq& n^22^{n-\alpha(x)}\left(\frac{|x-\bar{x}|}{R}\right)^{\alpha(x)}.
\end{eqnarray*}
Subsequently, we estimate  $I_2,\ I_3,\ I_4$ and $I_5$ in the following manner  
\begin{eqnarray*} 
|I_2| &\leq& \frac{1}{n\omega_n}\int_{\partial B_2}|\bar{x}-y|^{1-n}dS(y)\leq\frac{1}{n\omega_n}\int_{\partial B_2}R^{1-n}dS(y)=2^{n-1},	\\
|I_3| &\leq& \int_{B_{\delta}(\zeta )\cap B_2}|D_{ij}\Gamma(x-y)|\,|f(x)-f(y)|dy\leq\frac{1}{\omega_n}\int_{B_{\delta}(\zeta )\cap B_2}|x-y|^{-n+\alpha(x)}\frac{|f(x)-f(y)|}{|x-y|^{\alpha(x)}}dy	\\
&\leq& \frac{1}{\omega_n}[f]_{\alpha(\cdot), x}\int_{B_\frac{3\delta}{2}(x)}|x-y|^{\alpha(x)-n}dy\leq\frac{n}{\alpha(x)}\left(\frac{3\delta}{2}\right)^{\alpha(x)}[f]_{\alpha(\cdot), x},\\
|I_4|&\leq& \frac{n}{\alpha(x)}\left(\frac{3\delta}{2}\right)^{\alpha(\bar{x})}[f]_{\alpha(\cdot), \bar{x}},
\end{eqnarray*}
\begin{eqnarray*}
|I_5|&=&\left|\int_{\partial(B_2-B(\zeta,\delta ))}D_i\Gamma(x-y)\nu_j(y)dS(y)\right| \leq \left|\int_{\partial B_2}D_i\Gamma(x-y)\nu_j(y)dS(y)\right|	\\
&+&\left|\int_{\partial B(\zeta,\delta)}D_i\Gamma(x-y)\nu_j(y)dS(y)\right| \leq  2^{n-1}+\frac{1}{n\omega_n}\int_{\partial B(\zeta,\delta )}|x-y|^{1-n}dS(y)	\\
&\leq& 2^{n-1}+\frac{1}{n\omega_n}\left(\frac{\delta}{2}\right)^{1-n}\int_{\partial B(\zeta,\delta)}dS(y)=2^n.
\end{eqnarray*} 
Since $|\bar{x}-y|\leq\frac{3}{2}|\zeta-y|\leq 3|tx+(1-t)\bar{x}-y|$, we get
\begin{eqnarray*}
|I_6|&=&\left|\int_{B_2\setminus B(\zeta,\delta)}\int_0^1\frac{d}{dt}D_{ij}\Gamma(tx+(1-t)\bar{x}-y)dt(f(\bar{x})-f(y))dy\right|		\\
&\leq& \int_{B_2\setminus B(\zeta,\delta)}\sup_{t\in[0,\ 1]}|DD_{ij}\Gamma(tx+(1-t)\bar{x}-y)|\,|x-\bar{x}|\,|f(\bar{x})-f(y)|dy 	\\
&\leq&  c\delta\int_{|y-\zeta|\geq\delta}\sup_{t\in[0,\ 1]}\frac{|f(\bar{x})-f(y)|}{|tx+(1-t)\bar{x}-y|^{n+1}}dy	\\
&\leq& c\delta[f]_{\alpha(\cdot),\bar{x}}\int_{|y-\zeta|\geq\delta}\sup_{t\in[0,\ 1]}\frac{|\bar{x}-y|^{\alpha(\bar{x})}}{|tx+(1-t)\bar{x}-y|^{n+1}}dy	\\
&\leq& c\delta[f]_{\alpha(\cdot),\bar{x}}3^{\alpha(\bar{x})}2^{n+1-\alpha(\bar{x})}\int_{|y-\zeta|\geq\delta}|\zeta-y|^{\alpha(\bar{x})-n-1}dy\leq\frac{cn}{1-\alpha(\bar{x})}\delta^{\alpha(\bar{x})}[f]_{\alpha(\cdot),\bar{x}}3^{\alpha(\bar{x})}2^{n+1-\alpha(\bar{x})},
\end{eqnarray*}
where $c=\frac{n(n+5)}{\omega_n}$.
Finally, we obtain
\begin{eqnarray*}
  |D_{ij}w(\bar{x})&-& D_{ij}w(x)|	\\
  &\leq&  C\left(|f(x)|\frac{\delta^{\alpha(x)}}{R^{\alpha(x)}}+|f(x)-f(\bar{x})|+|f|_{\alpha(\cdot),x}\delta^{\alpha(x)}+|f|_{\alpha(\cdot),\bar{x}}\delta^{\alpha(\bar{x})}\right).
\end{eqnarray*}
Now, we divide the above inequality by $\delta^{\alpha(x)}$ and multiply by $R^{\alpha(x)}$:
\begin{align*}
  R^{\alpha(x)}&\frac{|D_{ij}w(\bar{x})-D_{ij}w(x)|}{|x-\bar{x}|^{\alpha(x)}}\\
  &\leq C\left(|f(x)|+R^{\alpha(x)}\frac{|f(x)-f(\bar{x})|}{|x-\bar{x}|^{\alpha(x)}}+R^{\alpha(x)}|f|_{\alpha(\cdot),x}+R^{\alpha(x)-\alpha(\bar{x})}R^{\alpha(\bar{x})}|f|_{\alpha(\cdot),\bar{x}}\delta^{\alpha(\bar{x})-\alpha(x)}\right).
\end{align*}
Since $\alpha$ is $\log$-H\"older continuous, we get that $\delta^{\alpha(\bar{x})-\alpha(x)}$ is bounded. Namely, we have
\begin{align*}
  \delta^{\alpha(\bar{x})-\alpha(x)}=|\bar{x}-x|^{\alpha(\bar{x})-\alpha(x)}\leq e^{\ln|\bar{x}-x|(\alpha(\bar{x})-\alpha(x))}\leq
  e^{\left|\ln|\bar{x}-x|\right||\alpha(\bar{x})-\alpha(x)|}\leq e^{c_{\log}(\alpha)}.
\end{align*}
The term $R^{\alpha(x)-\alpha(\bar{x})}$ is bounded as well. Indeed, 
\begin{itemize}
\item  If $R \geq 1$, then
\begin{align*}
  R^{\alpha(x)-\alpha(\bar{x})}\leq R^{2\alpha^+}\leq d^{2\alpha^+},
\end{align*}
where $d=\hbox{diam}\left(\Omega\right)$,
\item If $R<1$ and $\alpha(x)-\alpha(\bar{x}) \geq 0$, then
\begin{align*}
  R^{\alpha(x)-\alpha(\bar{x})}\leq 1,
\end{align*}
\item If $R<1$ and $\alpha(x)-\alpha(\bar{x}) < 0$, then
\begin{align*}
  R^{\alpha(x)-\alpha(\bar{x})}\leq \frac{1}{2^{\alpha(x)-\alpha(\bar{x})}} \delta^{\alpha(\bar{x})-\alpha(x)}\leq 2^{\alpha^{+}-\alpha^{-}}e^{c_{\log}(\alpha)}.
\end{align*}
\end{itemize}
This finishes the proof of Lemma \ref{lem1}. 
\end{proof}

As a corollary we get the following claim.
\begin{theo}\label{tw2}
Let $\Omega$ be a bounded domain in $\mathbb{R}^n$ and $f\in C^{\alpha(\cdot)}(\Omega)$. If $u\in C^{2}(\Omega)$ 
satisfies  $\Delta u=f$ in $\Omega$, 
then for any concentring balls $B_1=B(x_0, R)$, $B_2=B(x_0, 2R)\subset\subset\Omega$, we have 
\begin{eqnarray*}
|u|'_{2,\alpha(\cdot),B_1}\leq C\left(|u|_{0,B_2}+R^2|f|'_{0,\alpha(\cdot),B_2}\right),
\end{eqnarray*}
where $C=C(\textup{diam}(\Omega), n, \alpha^-, \alpha^+, c_{\log}(\alpha)).$
\end{theo}

\begin{proof} Let us write $u=v+w$, where $v$ is a harmonic function in $B_2$ and $w$ is the Newtonian potential of $f$ on $B_2$. 
By the basic properties of the Newtonian potential we get
\begin{eqnarray*}
|Dw(x)| &=& \left|\int_{B_2}D\Gamma(x-y)f(y)dy\right| \leq |f|_{0,B_2}\frac{1}{\omega_n}\int_{B_2}|x-y|^{1-n}dy\\
&\leq& |f|_{0,B_2}\frac{1}{\omega_n}\int_{B(x,3R)}|x-y|^{1-n}dy\leq CR|f|_{0,B_2}.
\end{eqnarray*}
In view of  Lemma \ref{lem1}, we conclude 
\begin{equation}\label{tw2in1}
R|Dw|_{0,B_1}+R^2\left|D^2w\right|'_{0,\alpha(\cdot),B_1}\leq CR^2|f|'_{0,\alpha(\cdot),B_2}.
\end{equation}
Subsequently, 
\begin{eqnarray*}
R^{\alpha(x)}\left|D^2v(x)-D^2v(y)\right| &=& R^{\alpha(x)}\left|\int_0^1\frac{d}{dt}D^2v((1-t)x+ty)dt\right| \leq R^{\alpha(x)}\left|D^3v\right|_{0,B_1}|x-y| 	\\
 &\leq&  C R^{\alpha(x)-3}|v|_{0,B_2}|x-y|,
\end{eqnarray*}
where the last inequality is obtained from estimation for harmonic functions 
$$\sup\limits_{B_1}\left|D^3v\right| \leq \left(\frac{3n}{R}\right)^3\sup\limits_{B_2}|v|. $$
Therefore, we have
\begin{eqnarray*}
R^{\alpha(x)}\frac{\left|D^2v(x)-D^2v(y)\right|}{|x-y|^{\alpha(x)}} \leq C R^{\alpha(x)-3}|v|_{0,B_2}|x-y|^{1-\alpha(x)}\leq C R^{-2}|v|_{0,B_2},
\end{eqnarray*}
and
\begin{equation*}
R|Dv|_{0,B_1}+R^2\left|D^2v\right|'_{0,\alpha(\cdot),B_1} \leq C|v|_{0,B_2}\leq C\left(|u|_{0,B_2}+R^2|f|_{0,B_2}\right)
\end{equation*}
where the last inequality is a consequence of the identity $v=u-w$ and  the inequality $|w|_{0,B_2}\leq CR^2|f|_{0,B_2}$ when $n>2$. 
For $n=2$, we can write $u(x_1,x_2,x_3)=u(x_1,x_2)$ and consider $u$ as a~solution of the Poisson's equation in a ball in $\mathbb{R}^3$. 
This completes the proof of the theorem.
\end{proof}
In order to develop the theory, we need to introduce some notations. For $x, y \in\Omega$ let us denote the following distances to the boundary, $d_x= \hbox{dist}(x,\partial\Omega)$ and $d_{x,y}=\min\{d_x, d_y\}$. Now, we define useful norms and seminorms:
\begin{eqnarray*}
[u]_{k,0,\Omega}^* =[u]_{k,\Omega}^*=\sup_{\substack{x\in\Omega\\|\beta|=k}}d_x^k|D^{\beta} u(x)|, \quad [u]^{*}_{k,\alpha(\cdot),\Omega}=\sup_{\substack{x,y\in\Omega\\|\beta|=k}}d^{k+\alpha(x)}_{x,y}\frac{|D^{\beta}u(x)-D^{\beta}u(y)|}{|x-y|^{\alpha(x)}},\\
|u|^*_{k,\Omega}=\sum\limits_{j=0}^k[u]_{j,\Omega}^*,\quad
|u|_{k,\alpha(\cdot),\Omega}^*=|u|_{k,\Omega}^*+[u]_{k,\alpha(\cdot)}^*,\\
\end{eqnarray*}
\begin{eqnarray*}
[u]_{k,\Omega}^{(s)}=[u]_{k,0,\Omega}^{(s)}=\sup_{\substack{x\in\Omega \\ |\beta|=k}}d_x^{k+s}|D^{\beta}u(x)|,\quad
[u]_{k,\alpha(\cdot),\Omega}^{(s)}=\sup_{\substack{x,y\in\Omega\\ |\beta|=k}}d^{k+\alpha(x)+s}_{x,y}\frac{|D^{\beta}u(x)-D^{\beta}u(y)|}{|x-y|^{\alpha(x)}},\\
|u|^{(s)}_{k,\Omega}=\sum\limits_{j=0}^k[u]^{(s)}_{j,\Omega}, \quad |f|^{(s)}_{k,\alpha(\cdot),\Omega}=|f|_{k,\Omega}^{(s)}+[f]^{(s)}_{k,\alpha(\cdot),\Omega}.
\end{eqnarray*}

\begin{theo}\label{twlap}
Let $u\in C^2(\Omega)$, $f\in C^{\alpha(\cdot)}(\bar{\Omega})$ satisfy $\Delta u=f$ in an open bounded set $\Omega\subset\mathbb{R}^n$. 
Then, the following inequality
\begin{eqnarray*}
|u|^{*}_{2,\alpha(\cdot),\Omega}\leq C\left(|u|_{0,\Omega}+|f|^{(2)}_{0,\alpha(\cdot),\Omega}\right)
\end{eqnarray*}
is satisfied, where $C=C(\textup{diam}(\Omega), n, \alpha^-, \alpha^+, c_{\log}(\alpha))$.
\end{theo}

\begin{proof}
Let $x\in\Omega$ and $3R=d_x$. Then, by Theorem \ref{tw2}, we get
\begin{eqnarray*}
d_x|Du(x)|+d^2_x\left|D^2u(x)\right| &\leq& 3R|Du|_{0,B_1}+(3R)^2\left|D^2u\right|_{0,B_1} \leq C\left(|u|_{0,B_2}+R^2|f|'_{0,\alpha(\cdot),B_2}\right)	\\
  &\leq& C\left(|u|_{0,\Omega}+|f|^{(2)}_{0,\alpha(\cdot),\Omega}\right).
\end{eqnarray*}
Next, let $x,y\in\Omega$ and $d_x\leq d_y$.
For $y\in B_1$ we have 
\begin{eqnarray*}
d_{x,y}^{2+\alpha(x)}\frac{\left|D^2u(x)-D^2u(y)\right|}{|x-y|^{\alpha(x)}} = (3R)^{2+\alpha(x)}\frac{\left|D^2u(x)-D^2u(y)\right|}{|x-y|^{\alpha(x)}} \leq CR^2\left[D^2u\right]'_{\alpha(\cdot),B_1}.
\end{eqnarray*}
When $y\notin B_1$, we obtain
\begin{eqnarray*}
d_{x,y}^{2+\alpha(x)}\frac{\left|D^2u(x)-D^2u(y)\right|}{|x-y|^{\alpha(x)}} \leq \frac{(3R)^{\alpha(x)+2}}{R^{\alpha(x)}}\left(\left|D^2u(x)\right|+\left|D^2(y)\right|\right).
\end{eqnarray*}
Finally, we have
\begin{eqnarray*}
d_{x,y}^{2+\alpha(x)}&&\hspace{-8mm} \frac{\left|D^2u(x)-D^2u(y)\right|}{|x-y|^{\alpha(x)}} \leq CR^2\left(\left|D^2u(x)\right|+\left|D^2(y)\right|+\left[D^2u\right]'_{\alpha(\cdot),B_1}\right)	\\
&\leq& C\left(|u|'_{2,\alpha(\cdot),B_1} +[u]^*_{2,\Omega}\right)\leq C\left(|u|_{0,B_2}+R^2|f|'_{0,\alpha(\cdot),B_2}+[u]^*_{2,\Omega}\right)	\\
&\leq& C\left(|u|_{0,\Omega}+|f|^{(2)}_{0,\alpha(\cdot),\Omega}\right).
\end{eqnarray*}
Thus, the proof is finished.
\end{proof}

For simplicity of notation, we write
$B^+_2=B_2\cap\mathbb{R}^n_+$ and $B^+_1=B_1\cap\mathbb{R}^n_+$, $T=\{x_n=0\}$. Furthermore, we assume that the center of the ball $B$, $x_0\in\bar{\mathbb{R}}^n_+$.
\begin{lem}\label{lembrz}
Let $f\in C^{\alpha(\cdot)}\left(\bar{B}_2^+\right)$ and $w$ be the Newtonian potential of $f$ in $B_2^+$ and let us assume that $B_2^+\subset\Omega$, 
where $\Omega\subset\mathbb{R}^n$ is an open and bounded set. Then, $w\in C^{2,\alpha(\cdot)}\left(\bar{B}_1^+\right)$ 
and the following inequality is satisfied
\begin{equation*}
\left|D^2w\right|'_{0,\alpha(\cdot),B_1^+}\leq C|f|'_{0,\alpha(\cdot),B_2^+},
\end{equation*}
where $C=C(\textup{diam}(\Omega), n, \alpha^-, \alpha^+, c_{\log}(\alpha))$.
\end{lem}

\begin{proof}
We can assume that $B_2$ intersects $T$. Moreover, we assume that $i\neq n$ or $j\neq n$, then  
\begin{equation*}
\int_{\partial B_2^+ \cap T}D_j\Gamma(x-y)\nu_i(y)dS(y) = \int_{\partial B_2^+ \cap T}D_i\Gamma(x-y)\nu_j(y)dS(y)
\end{equation*}
vanishes, since $\nu_i=0$ or $\nu_j=0$. We shall prove similar estimates as in the proof of Lemma~\ref{lem1}.
For $x\in B^+_1$, we have
\begin{equation*}
D_{ij}w(x) = \int_{B^+_2}D_{ij}\Gamma(x-y)(f(y)-f(x))dy-f(x)\int_{\partial B^+_2}D_i\Gamma(x-y)\nu_j(y)dS(y).
\end{equation*}
Thus, we get
\begin{eqnarray*}
|D_{ij}w(x)|&\leq&\frac{|f(x)|}{n\omega_n}R^{1-n}\int_{\partial B^+_2 \setminus T}dS(y)+\frac{[f]_{\alpha(\cdot), x}}{\omega_n}\int_{B^+_2}|x-y|^{\alpha(x)-n}dy	\\
&\leq&2^{n-1}|f(x)|+\frac{[f]_{\alpha(\cdot), x}}{\omega_n}\int_{B(x, 3R)}|x-y|^{\alpha(x)-n}dy \\
&\leq& 2^{n-1}|f(x)|+[f]_{\alpha(\cdot), x}(3R)^{\alpha(x)}\frac{n}{\alpha(x)}	
\leq C_1\left(|f(x)|+[f]_{\alpha(\cdot), x}\frac{(3R)^{\alpha(x)}}{\alpha(x)}\right).
\end{eqnarray*}
Now, let us fix $\bar{x}\in B_1$ and let us denote $\delta=|x-\bar{x}|$, $\zeta=\frac{1}{2}(x+\bar{x})$. Then, we can write 
$$
D_{ij}w(\bar{x})-D_{ij}w(x)=f(x)J_1+(f(x)-f(\bar{x}))J_2+J_3+J_4+(f(x)-f(\bar{x}))J_5+J_6,
$$
where $J_1,\ J_2,\ J_3,\ J_4,\ J_5,\ J_6$ have the same form like $I_1,...,I_6$ in the proof of Lemma \ref{lem1} with $B_2$ replaced by $B^+_2$. Moreover, each of 
$J_1, ...,  J_6$ can be estimated in similar manner as in the proof of Lemma \ref{lem1}. When $i=j=n$, we use the fact that $u$ satisfies the Poisson's equation, so
\begin{align*}
u_{nn}=\Delta u-\sum_{i=1}^{n-1}u_{ii}=f-\sum_{i=1}^{n-1}u_{ii}.
\end{align*}
Hence, by the previous inequality we can easily estimate the term $u_{nn}$.
\end{proof}
\begin{theo}\label{analtw2}
Let $\Omega\subset\mathbb{R}^n$ be an open and bounded set and ${B}_2^+\subset\Omega$. 
If $u\in C^2(B_2^+)\cap C^0(\bar{B}^+_2)$, $f\in C^{\alpha(\cdot)}(\bar{B}^+_2)$, 
satisfy $\Delta u=f$ in $B^+_2$ and $u=0$ on $T$, then $u\in C^{2,\alpha(\cdot)}(\bar{B}^+_1)$. 
Moreover, the following estimate
\begin{eqnarray*}
|u|'_{2, \alpha(\cdot), B_1^+}\leq C\left(|u|_{0, B_2^+}+R^2|f|'_{0,\alpha(\cdot),B^+_2}\right)
\end{eqnarray*}
holds, where $C=C(\textup{diam}(\Omega), n, \alpha^-, \alpha^+, c_{\log}(\alpha))$.
\end{theo}
\begin{proof}
For $x=(x',-x_n)\in\mathbb{R}^n$, where $x'=(x_1,\ldots,x_{n-1})$ we define $x^*=(x',-x_n)$. 
We set $B_2^-=\{x^*\colon x\in B_2^+\}$ and $D=B^+_2\cup B^-_2\cup \left(B_2\cap T\right)$. Next, we define the reflection map and the reflection exponent
\begin{align*}
f^*(x)=\left\{\begin{array}{ll}
f(x),&\textrm{ if $x\in B_2^+\cup(B_2\cap T)$}\\
f(x^*),&\textrm{ if $x\in B_2^-$}\end{array}\right.&&\alpha^*(x)=\left\{\begin{array}{ll}
\alpha(x),&\textrm{ if $x\in B_2^+\cup(B_2\cap T)$}\\
\alpha(x^*),&\textrm{ if $x\in B_2^-$.}\end{array}\right.
\end{align*} 
First of all, we show that $\alpha^*\in\mathcal{A}^{\log}(D)$. Let us assume that $|x-y|\leq \frac{1}{2}$. If $x,y\in B^+_2\cup T$ or $x,y\in B^-_2$, then it is clear that 
\begin{align*}
|\alpha^*(x)-\alpha^*(y)|\left|\ln|x-y|\right|\leq c_{\log}(\alpha).
\end{align*}
Let us assume that $x\in B^+_2\cup T$ and $y\in B^-_2$. 
Since
\begin{align*}
|x-y^*|\leq|x-y|,
\end{align*}
we get 
\begin{align*}
|\alpha^*(x)-\alpha^*(y)|\left|\ln|x-y|\right|\leq|\alpha(x)-\alpha(y^*)|\left|\ln|x-y^*|\right|\leq c_{\log}(\alpha).
\end{align*}
Now, assuming that $|x-y|>\frac{1}{2}$, we obtain
\begin{align*}
|\alpha^*(x)-\alpha^*(y)|\left|\ln|x-y|\right|\leq2\alpha^+\max\left\{|\ln\frac{1}{2}|, \ln2\hbox{diam}(\Omega)\right\}.
\end{align*}
Next, we prove  that $f^*\in C^{\alpha^*(\cdot)}(D)$ and $|f^*|'_{0,\alpha^*(\cdot),D}\leq |f|'_{0,\alpha(\cdot), B_2^+}$. 
If $x,y\in B^+_2$ or $x,y\in B^-_2$ then, it is easy to see that
\begin{align*}
R^{\alpha^*(x)}\frac{|f^*(x)-f^*(y)|}{|x-y|^{\alpha^*(x)}}\leq[f]'_{0,\alpha(\cdot), B_2^+}.
\end{align*}
If $x\in B^+_2\cup T$ and $y \in B^-_2$, thus we get
\begin{align*}
R^{\alpha^*(x)}\frac{|f^*(x)-f^*(y)|}{|x-y|^{\alpha^*(x)}}\leq R^{\alpha^*(x)}\frac{|f(x)-f(y^*)|}{|x-y^*|^{\alpha(x)}}\leq[f]'_{0,\alpha(\cdot)B_2^+}.
\end{align*}
So we obtain
\begin{align*}
[f^*]_{0,\alpha^*(\cdot), D}\leq[f]'_{0,\alpha(\cdot)B_2^+},
\end{align*}
and this leads us to the following inequality
\begin{align*}
|f^*|'_{0,\alpha^*(\cdot),D}\leq |f|'_{0,\alpha(\cdot), B_2^+}.
\end{align*}
Let 
$$
w(x)=\int_{B^+_2}\left(\Gamma(x-y)-\Gamma(x^*-y)\right)f(y)dy.
$$
It is easy to see that $w(x',0)=0$ and that for $x\in B_2^+$ we have $\Delta w(x)=f$ since
$$
\Delta \int_{B^+_2}\Gamma(x-y)f(y)dy=f(x)
$$
and
$$
\Delta\int_{B^+_2}\Gamma(x^*-y)f(y)dy=\int_{B^+_2}\Gamma(x-y^*)f(y)dy=0.
$$
Moreover, we see that
$$
\int_{B^+_2}\Gamma(x-y^*)f(y)dy=\int_{B^-_2}\Gamma(x-y)f^*(y)dy,
$$
thus, we get
\begin{eqnarray*}
w(x)=2\int_{B^+_2}\Gamma(x-y)f(y)dy-\int_D\Gamma(x-y)f^*(y)dy.
\end{eqnarray*}
Let us denote 
\begin{eqnarray*}
    w^*(x)=\int_D\Gamma(x-y)f^*(y)dy, \quad \bar{w}(x)=\int_{B_2}\Gamma(x-y)f^*(y)dy.
\end{eqnarray*}
Observe that, by Lemma \ref{lem1}, we have
\begin{eqnarray}\label{nrow1}
\left|D^2\bar{w}\right|'_{0,\alpha(\cdot),B_1^+} \leq C|f^*|'_{0,\alpha^*(\cdot),D}\leq C|f|'_{0,\alpha(\cdot),B_2^+}.
\end{eqnarray}
There exists a harmonic function $g$ on $B_2$ such that we have $w^*=g+\bar{w}$. Let us assume that $n>2$, thus by estimates for harmonic function we have
\begin{align*}
|D^2g|'_{0,\alpha(\cdot),B_1^+}\leq CR^{-2} |g|_{0,B_2}\leq CR^{-2}\left(|w^*|_{0,B_2}+|\bar{w}|_{0,B_2}\right).
\end{align*}
Let $x\in B_2$, then
\begin{align}\label{pomrow1}
|w^*(x)|\leq C|f^*|_{0,D}\int_{D}|\Gamma(x-y)|dy\leq C|f|_{0,B^+_2}\int_{B(x,6R)}|\Gamma(x-y)|dy=CR^2|f|_{0,B^+_2}.
\end{align}
Gathering (\ref{pomin1}) with inequality $|\bar{w}|_{0,B_2}\leq C R^2|f|_{0,B_2^+}$ yields 
\begin{align*}
|D^2g|'_{0,\alpha(\cdot),B_1^+}\leq C |f|_{0,B_2^+}.
\end{align*}
Hence, thanks to inequality (\ref{nrow1}) we obtain
\begin{align*}
|D^2w^*|'_{0,\alpha(\cdot),B_1^+}\leq C |f|'_{0,B_2^+}.
\end{align*}
Consequently, by Lemma \ref{lembrz} we have
\begin{align}\label{nrow2}
\left|D^2w\right|'_{0,\alpha(\cdot),B^+_1}\leq C|f|'_{0,\alpha(\cdot),B_2^+}.
\end{align}
Furthermore, by similar consideration as in (\ref{pomrow1}) we get the following inequality
\begin{align}\label{nrow3}
|w|_{0,B_1^+}\leq CR^2|f|_{0,B^+_2}.
\end{align}
Next, for $x\in B_1^+$ we get
\begin{align*}
|Dw(x)|\leq2\int_{B^+_2}|D\Gamma(x-y)||f(y)|dy-\int_D|\Gamma(x-y)|f^*(y)|dy\leq C|f|_{0, B^+_2}\int_{B(x, 6R)}|x-y|^{1-n}dy=CR|f|_{0, B^+_2}.
\end{align*}
In view of the above inequality with (\ref{nrow2}) and (\ref{nrow3}), we conclude
\begin{align*}
|w|'_{2,\alpha(\cdot),B_1^+}\leq C R^2|f|'_{0,\alpha(\cdot), B_2^+}.
\end{align*}
Finally, we put $v=u-w$, then $v$ is a harmonic function and $v=0$ on $T$. Hence, by reflection it can be extended to a harmonic function in the whole $B_2$. 
Thus, by estimates of harmonic functions as in the proof of Theorem~\ref{tw2}, we see that 
\begin{align*}
\left|v\right|'_{2,\alpha(\cdot),B^+_1} \leq C|v|_{0,B_2}=C|v|_{0,B^+_2}\leq C\left(|u|_{0,B^+_2}+R^2|f|_{0,B^+_2}\right)
\end{align*}
and the proof follows. When $n=2$, we can proceed in the same manner as in the end of the proof of Theorem~\ref{tw2}.
\end{proof}

Let $\Omega$ be an open subset of $\mathbb{R}^n$ and $T\subset \partial \Omega$.  For $x,\ y\in\Omega$ we define $\bar{d}_x=\hbox{dist}(x,\ \partial\Omega \setminus T)$, $\bar{d}_{x,y}=~\min(\bar{d}_x,\ \bar{d}_y)$. 
In the sequel, we shall use the following notation
\begin{eqnarray*}
[u]^*_{k,0,\Omega\cup T}=[u]^*_{k,\Omega\cup T} = \sup_{\substack{x\in\Omega
 |\beta|=k}}\bar{d}_x^k\left|D^{\beta}(x)\right|,\quad 	[u]^*_{k,\alpha(\cdot),\Omega \cup T}=\sup_{\substack{x,y\in\Omega |\beta|=k}} \bar{d}_{x,y}^{k+\alpha(x)}\frac{\left|D^{\beta}u(x)-D^{\beta}u(y)\right|}{|x-y|^{\alpha(x)}},	\\
|u|^*_{k,0,\Omega\cup T}=|u|^*_{k,\Omega\cup T}=\sum\limits_{j=0}^k[u]_{j,\Omega\cup T}^*,\quad		|u|_{k,\alpha(\cdot),\Omega\cup T}=|u|^*_{k,\Omega\cup T}+[u]_{k,\alpha(\cdot),\Omega\cup T}^*,\\
|u|_{0,\alpha(\cdot),\Omega\cup T}=\sup_{x\in\Omega}\bar{d}^k_x|u(x)|+\sup_{x,y\in\Omega}\bar{d}^{k+\alpha(x)}_{x,y}\frac{|u(x)-u(y)|}{|x-y|^{\alpha(x)}}.
\end{eqnarray*}

\begin{theo}\label{laplthbr}
Let $\Omega$ be an open and bounded subset of $\mathbb{R}^n_+$ with boundary portion $T$ on $\{x_n=0\}$ and let 
$u\in C^2\left(\Omega\right)\cap C^0\left({\Omega\cup T}\right),\ f\in C^{\alpha(\cdot)}(\bar{\Omega})$ satisfy $\Delta u=f$ and $u=0$ on $T$. Then
\begin{equation*}
|u|^*_{2,\alpha(\cdot),\Omega\cup T}\leq C\left(|u|_{0,\Omega}+|f|^{(2)}_{0, \alpha(\cdot), \Omega\cup T}\right),
\end{equation*}
where $C=C(\textup{diam}(\Omega), n, \alpha^-, \alpha^+, c_{\log}(\alpha))$.
\end{theo}

\begin{proof}
We proceed analogously to the proof of Theorem \ref{twlap}.
Let $x\in\Omega$ and $3R=\bar{d}_x$. From Theorem \ref{analtw2} we have the following inequalities 
\begin{eqnarray*}
\bar{d}_x|Du(x)|+\bar{d}^2_x\left|D^2u(x)\right| &\leq& 3R|Du|_{0,B^+_1}+(3R)^2\left|D^2u\right|_{0,B^+_1} \leq C\left(|u|_{0,B^+_2}+R^2|f|'_{0,\alpha(\cdot),B^+_2}\right)	\\
  &\leq& C\left(|u|_{0,\Omega\cup T}+|f|^{(2)}_{0,\alpha(\cdot),\Omega\cup T}\right).
\end{eqnarray*}
Now, let $x,y\in\Omega$ and $\bar{d}_x\leq \bar{d}_y$.
For $y\in B^+_1$, we get
\begin{eqnarray*}
\bar{d}_{x,y}^{2+\alpha(x)}\frac{\left|D^2u(x)-D^2u(y)\right|}{|x-y|^{\alpha(x)}} = (3R)^{2+\alpha(x)}\frac{\left|D^2u(x)-D^2u(y)\right|}{|x-y|^{\alpha(x)}} \leq CR^2\left[D^2u\right]'_{\alpha(\cdot),B_1}.
\end{eqnarray*}
When $y\notin B_1$, we have
\begin{eqnarray*}
\bar{d}_{x,y}^{2+\alpha(x)}\frac{\left|D^2u(x)-D^2u(y)\right|}{|x-y|^{\alpha(x)}} \leq \frac{(3R)^{\alpha(x)+2}}{R^{\alpha(x)}}\left(\left|D^2u(x)\right|+\left|D^2(y)\right|\right).
\end{eqnarray*}
Therefore, combining the above inequalities, we obtain 
\begin{eqnarray*}
\bar{d}_{x,y}^{2+\alpha(x)} &&\hspace{-8mm} \frac{\left|D^2u(x)-D^2u(y)\right|}{|x-y|^{\alpha(x)}} \leq CR^2\left(\left|D^2u(x)\right|+\left|D^2(y)\right|+\left[D^2u\right]'_{\alpha(\cdot),B^+_1}\right)	\\
&\leq& C\left(|u|'_{2,\alpha(\cdot),B^+_1} +[u]^*_{2,\Omega\cup T}\right)\leq C\left(|u|_{0,B^+_2}+R^2|f|'_{0,\alpha(\cdot),B^+_2}+[u]^*_{2,\Omega\cup T}\right)\\ 
&\leq& C\left(|u|_{0,\Omega\cup T}+|f|^{(2)}_{0,\alpha(\cdot),\Omega\cup T}\right).
\end{eqnarray*}
This completes the proof.
\end{proof}

\section{A priori estimates for fully elliptic equation}
In this section we study a priori estimates for solutions to the problem $Lu=f$, where the operator $L$ has the form 
\begin{equation*}
Lu=a^{ij}D_{ij}u+b^iD_iu+cu,
\end{equation*}
and the coefficients belong to the variable H\"{o}lder space $C^{\alpha(\cdot)}(\bar{\Omega})$.
The main aim of this section is to show the Schauder estimates in variable H\"{o}lder spaces. 

Let us start our considerations with operators with constant coefficients. 
\begin{lem}\label{lemful1}
Let $A= [A^{ij}]$ be a constant and symmetric matrix such that
\begin{equation}\label{koer}
	\lambda|\zeta|^2\leq A^{ij}\zeta_i\zeta_j\leq \Lambda|\zeta|^2, \forall_{\zeta\in\mathbb{R}^n}
\end{equation}
for certain positive constants $\lambda,\ \Lambda$. Now, let us define the operator $L_0u=A^{ij}D_{ij}u$.
\begin{itemize}
\item[(a)] Let $u\in C^2(\Omega)$, $f\in C^{\alpha(\cdot)}(\bar{\Omega})$ satisfy $L_0u=f$ in an open and bounded set $\Omega\subset\mathbb{R}^n$, then 
\begin{equation*}
	|u|^*_{2,\alpha(\cdot),\Omega}\leq C\left(|u|_{0,\Omega}+|f|_{0,\alpha(\cdot), \Omega}^{(2)}\right).
\end{equation*}
\item[(b)] Let $\Omega\subset\mathbb{R}^n_+$ be an open and bounded set with a part of boundary at $T=\{x\colon x_n=0\}$ and let 
$u\in C^2(\Omega)\cap C^0(\Omega\cup T),\ f\in C^{\alpha(\cdot)}(\bar{\Omega})$ satisfy $L_0u=f$ in $\Omega$, $u=0$ on $T$. Then
\begin{equation*}
	|u|^*_{2,\alpha(\cdot),\Omega\cup T}\leq C\left(|u|_{0,\Omega}+|f|^{(2)}_{0,\alpha(\cdot), \Omega\cup T}\right),
\end{equation*}
\end{itemize}
where $C=C(\textup{diam}(\Omega), n, \alpha^-, \alpha^+, c_{\log}(\alpha), \Lambda, \lambda)$.
\end{lem}
\begin{proof}
Let $P=[P^{ij}]$ be a nonsingular matrix with real entries and let us define $\tilde{u}(y)=u(P^{-1}y)$. We see that $\tilde{u}$ is defined 
on $\widetilde{\Omega}=P\Omega$. By straightforward calculation we obtain
\begin{equation*}
	D^2\tilde{u}(y)=P^{-T}D^2u(P^{-1}y)P^{-1}.
\end{equation*}
Hence, by the basic properties of traces we get
\begin{equation*}
	A^{ij}D_{ij}u(P^{-1}y)=\widetilde{A}^{ij}D_{ij}\tilde{u}(y),
\end{equation*}
where $\widetilde{A}=PAP^T$. We can take an orthogonal matrix $Q$ such that $QAQ^T$ is a diagonal matrix with eigenvalues $\lambda_1,\ldots,\lambda_n$. 
Let $D=[\lambda_i^{-\frac{1}{2}}\delta_{ij}]$, then $(DQ)A(DQ)^T$ is a unit matrix. Finally, let $R$ be a rotation such that $RDQ$ takes half space 
$\{x_n>0\}$ into $\{y_n>0\}$. Let us take $P=RDQ$. Then $\widetilde{A}$ is a unit matrix, thus we conclude that $\tilde{u}$ satisfies $\Delta \tilde{u}=\tilde{f}$, 
where $\tilde{f}=f\circ P^{-1}$ on $\widetilde{\Omega}$. We have got $P^{-T}P^{-1}=A$, so we see
\begin{align*}
	|P^{-1}x|^2=\left(P^{-1}x\right)^TP^{-1}x=x^TP^{-T}P^{-1}x=x^TAx.
\end{align*}
Thus, by condition (\ref{koer}) we obtain
\begin{align}\label{pomin1}
\Lambda^{-\frac{1}{2}}|x|\leq|Px|\leq\lambda^{-\frac{1}{2}}|x|.
\end{align}
Let us denote by $\tilde{d}_y=\hbox{dist}(y,\ \partial \widetilde{\Omega})$ and $\tilde{d}_{x,y}=\min\{\tilde{d}_x,\ \tilde{d}_y\}$. By virtue of (\ref{pomin1}), 
we obtain $\Lambda^{-\frac{1}{2}}d_x\leq\tilde{d}_{Px}\leq\lambda^{-\frac{1}{2}}d_x$. Let $\tilde{\alpha}=\alpha\circ P^{-1}$, then we have 
\begin{align*}
	[\tilde{u}]^*_{1,\tilde{\alpha}(\cdot),\widetilde{\Omega}}=\sup_{\substack{x\neq y\\ x,y\in\widetilde{\Omega}}}		\tilde{d}_{x,y}^{1+\tilde{\alpha}(x)}\frac{|D\tilde{u}(x)-D\tilde{u}(y)|}{|x-y|^{\tilde{\alpha}(x)}}=\\\sup_{	\substack{x\neq y\\ x,y\in\widetilde{\Omega}}}\tilde{d}_{x,y}^{1+\alpha(P^{-1}x)}\frac{|P^{-1}\left(Du(P^{-1}x)-Du(P^{-1}y)\right)|}{|x-y|^{\alpha(P^{-1}x)}}=
\sup_{\substack{x\neq y\\ x,y\in\Omega}}\tilde{d}_{Px,Py}^{1+\alpha(x)}\frac{|P^{-1}\left(Du(x)-Du(y)\right)|}{|P\left(x-y\right)|^{\alpha(x)}}.
\end{align*}
By the last equality, we obtain 
\begin{align*}
	C_2[u]^*_{1,\alpha(\cdot),\Omega}\leq[\tilde{u}]^*_{1,\tilde{\alpha}(\cdot),\widetilde{\Omega}}\leq C_1[u]^*_{1,\alpha(\cdot),\Omega}.
\end{align*}
In similar manner we get the subsequent inequalities
\begin{equation}\label{pomin2}
\begin{split}
C_2[u]^*_{k,\alpha(\cdot),\Omega}\leq[\tilde{u}]^*_{k,\tilde{\alpha}(\cdot),\widetilde{\Omega}}\leq C_1[u]^*_{k,\alpha(\cdot),\Omega}\\
C_2[u]^{(l)}_{k,\alpha(\cdot),\Omega}\leq[\tilde{u}]^{(l)}_{k,\tilde{\alpha}(\cdot),\widetilde{\Omega}}\leq C_1[u]^{(l)}_{k,\alpha(\cdot),\Omega},
\end{split}
\end{equation}
where $k, l=0, 1, 2, 3,\ldots$. Let us denote $PT$ as $\widetilde{T}$, then it is easy to see that we have
\begin{equation}\label{pomin3}
\begin{split}
C_2[u]^*_{k,\alpha(\cdot),\Omega\cup T}\leq[\tilde{u}]^*_{k,\tilde{\alpha}(\cdot),\widetilde{\Omega}\cup\widetilde{T}}\leq C_1[u]^*_{k,\alpha(\cdot),\Omega\cup T}\\
C_2[u]^{(l)}_{k,\alpha(\cdot),\Omega\cup T}\leq[\tilde{u}]^{(l)}_{k,\tilde{\alpha}(\cdot),\widetilde{\Omega}\cup\widetilde{T}}\leq C_1[u]^{(l)}_{k,\alpha(\cdot),\Omega\cup T},
\end{split}
\end{equation}
for $k, l=0, 1, 2, 3,\ldots$. Next, by Theorem \ref{twlap}, we conclude
\begin{equation*}
|\tilde{u}|^{*}_{2,\tilde{\alpha}(\cdot),\widetilde{\Omega}}\leq C\left(|\tilde{u}|_{0,\widetilde{\Omega}}+|\tilde{f}|^{(2)}_{0,\tilde{\alpha}(\cdot),\widetilde{\Omega}}\right),
\end{equation*}
therefore, by (\ref{pomin2}) we see that
\begin{equation*}
	|u|^{*}_{2,\alpha(\cdot),\Omega}\leq C\left(|u|_{0,\Omega}+|f|^{(2)}_{0,\alpha(\cdot),\Omega}\right).
\end{equation*}
This, finishes the proof of $(a)$.  Finally gathering Theorem \ref{laplthbr} with inequalities (\ref{pomin3}) we get $(b)$.
\end{proof}
We can now formulate our first main result. 
\begin{theo}[Interior Schauder estimates]\label{sches}
Let $\Omega\subset\mathbb{R}^n$ be an open and bounded set If $u\in C^{2,\alpha(\cdot)}(\bar{\Omega})$ satisfies
\begin{equation}\label{equel}
Lu=a^{ij}D_{ij}u+b^iD_iu+cu=f,
\end{equation}
where $f\in C^{\alpha(\cdot)}(\bar{\Omega})$ and there are positive constants $\lambda$ and $\Lambda$ such that
\begin{align*}
&a^{ij}(x)\zeta^i\zeta^j\geq\lambda|\zeta|^2\qquad\textrm{for all $x\in\Omega$ and for all $\zeta\in\mathbb{R}^n$,}\\
&|a^{ij}|^{(0)}_{0,\alpha(\cdot),\Omega},|b^i|_{0,\alpha(\cdot),\Omega}^{(1)},|c|^{(2)}_{0,\alpha(\cdot),\Omega}\leq \Lambda,
\end{align*} then
\begin{equation*}
|u|^*_{2,\alpha(\cdot), \Omega}\leq C\left(|u|_{0, \Omega}+|f|^{(2)}_{0,\alpha(\cdot), \Omega}\right),
\end{equation*}
where $C=C(\textup{diam}(\Omega), n, \alpha^-, \alpha^+, c_{\log}(\alpha), \Lambda, \lambda)$. 
\end{theo}
\begin{proof}
First of all, we shall prove the following crucial lemma.
\begin{lem}\label{lemsh}
Under the hypotheses of Theorem \ref{sches}, if moreover we take an open set $D\subset \Omega$ such that $2\alpha^-_D>\alpha^+_D$, then the following estimate
\begin{equation*}
|u|^*_{2,\alpha(\cdot), D}\leq C\left(|u|_{0, D}+|f|^{(2)}_{0,\alpha(\cdot), D}\right)
\end{equation*}
holds, where $C=C(\textup{diam}(\Omega), n, \alpha_D^-, \alpha_D^+, c_{\log}(\alpha), \Lambda, \lambda)$
\end{lem}
\begin{proof}
From Lemma \ref{intlem}, for all $\epsilon>0$ we have the following interpolation inequality
\begin{eqnarray}
 \label{intpom1} 
|u|^*_{2,D}\leq C(\epsilon)|u|_{0,D}+\epsilon[u]^*_{2,\alpha(\cdot), D}.
\end{eqnarray}
By virtue of the above inequality  it is sufficient to estimate the seminorm $[u]_{2,\alpha(\cdot),D}^*$.
Let us take two distinct points $x_0, y_0 \in D$ and let $\mu \leq \frac{1}{2}$, $d_{x_0}=d_{x_0, y_0}$ and $d=\mu d_{x_0}$, $B=B(x_0, d)$. 
We rewrite (\ref{equel}) as follows
\begin{eqnarray*}
	a^{ij}(x_0)D_{ij}u=(a^{ij}(x_0)-a^{ij})D_{ij}u-b^iD_iu-cu+f=F.
\end{eqnarray*}
We consider the above equation on the ball $B$. Thus, by Lemma \ref{lemful1} we obtain for $y_0\in B\left(x_0,\frac{d}{2}\right)$, the following inequality
\begin{eqnarray*}
	\left(\frac{d}{2}\right)^{2+\alpha(x_0)}\frac{|D^2u(x_0)-D^2u(y_0)|}{|x_0-y_0|^{\alpha(x_0)}}\leq C\left(|u|_{0,B}+|F|^{(2)}_{0,\alpha(\cdot),B}\right).
\end{eqnarray*}
This gives
\begin{eqnarray*}
	d_{x_0}^{2+\alpha(x_0)}\frac{|D^2u(x_0)-D^2u(y_0)|}{|x_0-y_0|^{\alpha(x_0)}}\leq \frac{C}{\mu^{2+\alpha_D^+}}\left(|u|_{0,B}+|F|^{(2)}_{0,\alpha(\cdot),B}\right).
\end{eqnarray*}
Next, if $y_0\notin B(x_0,\frac{d}{2})$, by straightforward calculations we obtain 
\begin{eqnarray*}
d_{x_0}^{2+\alpha(x_0)}\frac{|D^2u(x_0)-D^2u(y_0)|}{|x_0-y_0|^{\alpha(x_0)}}\leq d_{x_0}^{2+\alpha(x_0)}\left(\frac{2}{d}\right)^{\alpha(x_0)}|D^2u(x_0)-D^2u(y_0)|\\
\leq\left(\frac{2}{\mu}\right)^{\alpha(x_0)}\left(d_{x_0}^2|D^2u(x_0)|+d_{y_0}^2|D^2u(y_0)|\right)\leq\frac{4}{\mu^{\alpha_D^+}}[u]^*_{2, D}.
\end{eqnarray*}
Combining these inequalities yields
\begin{eqnarray}\label{schin}
	d_{x_0}^{2+\alpha(x_0)}\frac{|D^2u(x_0)-D^2u(y_0)|}{|x_0-y_0|^{\alpha(x_0)}}\leq\frac{C}{\mu^{2+\alpha_D^+}}\left(|u|_{0,B}+|F|^{(2)}_{0,\alpha(\cdot),B}\right)+\frac{4}{\mu^{\alpha_D^+}}[u]^*_{2,D}.
\end{eqnarray}
In order to finish the proof, we need to estimate the expression $|F|^{(2)}_{0,\alpha(\cdot),B}$. By the triangle inequality we have 
\begin{eqnarray}\label{inn1}
|F|^{(2)}_{0,\alpha(\cdot),B}\leq\sum_{i,j}|(a^{ij}(x_0)-a^{ij})D_{ij}u|^{(2)}_{0,\alpha(\cdot),B}+\sum_i|b^iD_iu|^{(2)}_{0,\alpha(\cdot),B}
+|cu|^{(2)}_{0,\alpha(\cdot),B}+|f|^{(2)}_{0,\alpha(\cdot), B}.
\end{eqnarray}
Now, we shall need the following result.
\begin{prop} \label{p1}
Let $g\in C^{\alpha(\cdot)}(D)$, then the following estimates hold
\begin{eqnarray}
[g]^{(2)}_{0,\alpha(\cdot), B}\leq 8\mu^{2+\alpha_D^-}[g]^{(2)}_{0,\alpha(\cdot), D},\label{schpom1}\\
|g|_{0,B}^{(2)}\leq 4\mu^2 |g|_{0,D}^{(2)} \label{schpom2}.
\end{eqnarray}
\end{prop}
\begin{proof}
We give the proof only the first inequality, the second one is left to the reader. 
Let us notice that for $x\in B$, we have $d_x \geq d_{x_0}-d=(1-\mu)d_{x_0}$. Let us take $x, y\in B$ and introduce the quantity 
$d_{x, y,B}=\min\{\hbox{dist}(x,\partial B),\ \hbox{dist}(y,\partial B)\}$. Hence, we get
\begin{eqnarray*}
d_{x,y,B}^{\alpha(x)+2}\frac{|g(x)-g(y)|}{|x-y|^{\alpha(x)}}\leq d^{\alpha(x)+2}\frac{|g(x)-g(y)|}{|x-y|^{\alpha(x)}}=
d^{\alpha(x)+2}d_{x,y}^{-\alpha(x)-2}d_{x,y}^{\alpha(x)+2}\frac{|g(x)-g(y)|}{|x-y|^{\alpha(x)}}\leq \\
 d^{\alpha(x)+2}\left((1-\mu)d_{x_0}\right)^{-\alpha(x)-2}[g]^{(2)}_{0,\alpha(\cdot), D} \leq 8\mu^{2+\alpha_D^-}[g]^{(2)}_{0,\alpha(\cdot), D}.
\end{eqnarray*}
This completes the proof of Proposition \ref{p1}.
\end{proof}
Now, we can estimate each of the terms in (\ref{inn1}). By the virtue of the above proposition we get
\begin{eqnarray*}
	|\left(a(x_0)-a\right)D^2u|^{(2)}_{0,\alpha(\cdot),B}\leq|a(x_0)-a|_{0,\alpha(\cdot),B}^{(0)}|D^2u|^{(2)}_{0,\alpha(\cdot),B}\leq\\
	|a(x_0)-a|_{0,\alpha(\cdot),B}^{(0)}\left(4\mu^2[u]^*_{2, D}+8\mu^{2+\alpha_D^-}[u]^*_{2,\alpha(\cdot),D}\right).
\end{eqnarray*}
Subsequently, we estimate the quantity $|a(x_0)-a|_{0,\alpha(\cdot),B}^{(0)}$ as follows
\begin{eqnarray*}
	|a(x_0)-a|_{0,\alpha(\cdot),B}^{(0)}=\sup_{x\in B}|a(x_0)-a(x)|+\sup_{\substack{x\neq y\\x,y\in B}}d_{x,y,B}^{\alpha(x)}\frac{|a(x)-a(y)|}{|x-y|^{\alpha(x)}}\leq
	2\sup_{\substack{x\neq y\\x,y\in B}}d^{\alpha(x)}\frac{|a(x)-a(y)|}{|x-y|^{\alpha(x)}}{\leq}\\ 
2\sup_{x,y\in D}d^{\alpha(x)}d_{x,y}^{-\alpha(x)}\sup_{\substack{x\neq y\\x,y\in D}}d_{x,y}^{\alpha(x)}\frac{|a(x)-a(y)|}{|x-y|^{\alpha(x)}}\leq 
4\mu^{\alpha_D^-}[a]^*_{0,\alpha(\cdot),D}\leq 4\Lambda\mu^{\alpha_D^-}.
\end{eqnarray*}
Next, by applying the interpolation inequality from Lemma \ref{intlem} with $\epsilon=\mu^{\alpha_D^-}$ we have
\begin{align}\label{inn2}
|\left(a(x_0)-a\right)D^2u|^{(2)}_{0,\alpha(\cdot),B}\leq 32 \Lambda\mu^{2+\alpha_D^-}\left([u]^*_{2,D}+\mu^{\alpha_D^-}[u]^*_{2,\alpha(\cdot),D}\right)\nonumber\\
{\leq}32 \Lambda\mu^{2+\alpha_D^-}\left(C|u|_{0, D}+2\mu^{\alpha_D^-}[u]^*_{2,\alpha(\cdot),D}\right).
\end{align}
Furthermore, using inequalities (\ref{schpom1}), (\ref{schpom2}) and Lemma \ref{intlem} we get
\begin{eqnarray}
|bDu|^{(2)}_{0,\alpha(\cdot),B} \leq 8\mu^2\Lambda\left(C|u|_{0,D}+\mu^{2\alpha_D^-}[u]^*_{2,\alpha(\cdot),D}\right),\label{inn3}\\
|cu|^{(2)}_{0,\alpha(\cdot),B} \leq 8\Lambda \mu^2\left(C|u|_{0,D}+\mu^{2\alpha_D^-}[u]^*_{2,\alpha(\cdot),D}\right), \label{inn4}\\
|f|^{(2)}_{0,\alpha(\cdot), B}\leq 8\mu^2|f|_{0,\alpha(\cdot),D}^{(2)}.\label{inn5}
\end{eqnarray}
Substituting inequalities (\ref{inn2})-(\ref{inn5}) into (\ref{inn1}) yields
\begin{equation}\label{schin1}
|F|^{(2)}_{0,\alpha(\cdot),B}\leq C\mu^{2+2\alpha_D^-}[u]^*_{2,\alpha(\cdot),D}+C_1\left(|u|_{0,D}+|f|^{(2)}_{0,\alpha(\cdot),D}\right).
\end{equation}
Therefore, combining (\ref{schin1}) with (\ref{schin}) we conclude that 
\begin{align*}
d_{x_0}^{2+\alpha(x_0)}\frac{|D^2u(x_0)-D^2u(y_0)|}{|x_0-y_0|^{\alpha(x_0)}}\leq \\C\mu^{2\alpha_D^--\alpha_D^+}[u]^*_{2,\alpha(\cdot),D}+
C_1\left(|u|_{0,D}+|f|^{(2)}_{0,\alpha(\cdot),D} \right)+\frac{4}{\mu^{\alpha_D^+}}[u]^*_{2,D}.
\end{align*}
Thus, according to interpolation inequality from Lemma \ref{intlem}, we have
\begin{align*}
d_{x_0}^{2+\alpha(x_0)}\frac{|D^2u(x_0)-D^2u(y_0)|}{|x_0-y_0|^{\alpha(x_0)}}\leq C\mu^{2\alpha_D^--\alpha_D^+}[u]^*_{2,\alpha(\cdot),D}+C_1\left(|u|_{0,D}+|f|^{(2)}_{0,\alpha(\cdot),D} \right).
\end{align*}
Finally, by assumption $2\alpha_D^->\alpha_D^+$ and since $\mu<\frac{1}{2}$ was arbitrary, we can take $\mu$ such that $C\mu^{2\alpha_D^--\alpha_D^+}<1$ and we can include the  $[u]^*_{2,\alpha(\cdot),D}$-term 
into the left-hand side of the above inequality. This completes the proof of Lemma \ref{lemsh}.
\end{proof}
We are now in a position to prove Theorem \ref{sches}. We shall use the covering argument. Let us take the balls $B_j=B(x_j, 2 r_j)$ 
such that on each of these the following inequality $2\alpha^->\alpha^+$ holds. 
Next, let us denote $\widetilde{B}_j= B(x_j, r_j)$. Since $\Omega$ is a relatively compact set, there exists an 
finite open covering $\{\widetilde{B}_j\}_{j=1}^N$ of $\bar{\Omega}$. 
Moreover, since the family $\{B_j\}^N_{j=1}$ is a covering of $\bar{\Omega}$, we can apply Lemma \ref{lemsh} to each of the sets $W_j=B_j\cap\Omega$ and we get
\begin{equation*}
[u]^*_{2,\alpha(\cdot), W_j}\leq C\left(|u|_{0,W_j}+|f|^{(2)}_{0,\alpha(\cdot),W_j}\right)\leq C\left(|u|_{0,\Omega}+|f|^{(2)}_{0,\alpha(\cdot),\Omega}\right).
\end{equation*}
Let us denote by $\delta<1$ a Lebesgue number of cover $\{\widetilde{B}_j\}$ and let $r=\min\{r_j\}$ (see \cite{gg}, \cite{gg1} for similar considerations). 
Then, if $x, y \in\Omega$ are such that $|x-y|\leq\delta$, there exists $j\in \{1,...,N\}$ such that $x, y\in \widetilde{B}_j$. 

Observe that, if $\partial B_j\cap\partial\Omega=\emptyset$, then we obtain the following estimate 
\begin{eqnarray}\label{thin1}
d_{x,y}^{\alpha(x)+2}\frac{|D^2u(x)-D^2u(y)|}{|x-y|^{\alpha(x)}}\leq d_{x,y,j}^{\alpha(x)+2}\frac{d^{2+\alpha(x)}}{r^{\alpha(x)+2}}\frac{|D^2u(x)-D^2u(y)|}{|x-y|^{\alpha(x)}}
\leq C(d,r)[u]^*_{2,\alpha(\cdot), W_j}\\ \leq C\left(|u|_{0,\Omega}+|f|^{(2)}_{0,\alpha(\cdot),\Omega}\right),\nonumber
\end{eqnarray}
where $d_{x,y,j}= \min\{\hbox{dist}(\partial {W}_j, x),\ \hbox{dist}(\partial {W}_j, y)\}$ and $d=\hbox{diam}(\Omega)$.
If $\partial B_j\cap\partial\Omega\neq\emptyset$, then there are two cases:

\begin{itemize}
\item[a)] $d_{x,y}= d_{x,y,j}$ It is clear that we have
$$
d_{x,y}^{\alpha(x)+2}\frac{|D^2u(x)-D^2u(y)|}{|x-y|^{\alpha(x)}}\leq C\left(|u|_{0,\Omega}+|f|^{(2)}_{0,\alpha(\cdot),\Omega}\right).
$$
\item[b)] $d_{x,y}>d_{x,y,j}$ Let us assume that $d_{x,y,j}=d_{x,j}$, then there exists $z\in\partial W_j\setminus\partial\Omega$ such that $|x-z|=d_{x,y,j}$. Hence,
\begin{align*}
d_{x,y,j}=|x-z|>|z-x_j|-|x-x_j|>2r-r=r.
\end{align*}
So, we conclude that $d_{x,y,j}>r$ and this leads us to the same estimation as inequality (\ref{thin1}).
\end{itemize}

Thus, for $|x-y|\leq\delta$ we have
\begin{equation}\label{theq1}
d_{x,y}^{\alpha(x)+2}\frac{|D^2u(x)-D^2u(y)|}{|x-y|^{\alpha(x)}}\leq C\left(|u|_{0,\Omega}+|f|^{(2)}_{0,\alpha(\cdot),\Omega}\right).
\end{equation}
Next, if $|x-y|>\delta$ then we get
\begin{eqnarray*}
  d_{x,y}^{\alpha(x)+2}\frac{|D^2u(x)-D^2u(y)|}{|x-y|^{\alpha(x)}}\leq d_{x,y}^{\alpha(x)+2}\delta^{-\alpha(x)}\left(|D^2u(x)|+|D^2u(y)|\right)\leq C[u]^*_{2, \Omega}.
\end{eqnarray*}
Combining the above inequality with Lemma \ref{intlem} yields
\begin{equation}\label{theq2}
d_{x,y}^{\alpha(x)+2}\frac{|D^2u(x)-D^2u(y)|}{|x-y|^{\alpha(x)}}\leq C\left(|u|_{0,\Omega}+\epsilon[u]^*_{2,\alpha(\cdot),\Omega}\right),
\end{equation}
for arbitrary $\epsilon>0$.
Finally, gathering (\ref{theq1}) and (\ref{theq2}), we get 
$$
d_{x,y}^{\alpha(x)+2}\frac{|D^2u(x)-D^2u(y)|}{|x-y|^{\alpha(x)}}\leq C\left(C(\varepsilon)|u|_{0,\Omega}+|f|^{(2)}_{0,\alpha(\cdot),\Omega}+\epsilon[u]^*_{2,\alpha(\cdot),\Omega}\right).
$$
Taking $\epsilon$ such that $C\epsilon<1$ we complete the proof of Theorem \ref{sches}.
\end{proof}
As a corollary we get the following lemma.
\begin{lem}\label{pomlem}
Let  $\Omega\subset \mathbb{R}^n$ be an open and bounded set and let $D\subset\subset \Omega$ and $d\leq \hbox{dist}(D, \partial\Omega)$. 
Let us also assume that $u\in C^{2,\alpha(\cdot)}(\Omega)$ and $f\in C^{\alpha(\cdot)}(\bar{\Omega})$ satisfy $Lu=f$, 
where $L$ is an elliptic operator with coefficients in $C^{\alpha(\cdot)}(\Omega)$, then the following inequality
\begin{align*}
d|Du|_{0,D}+d^2|D^2u|_{0,D}+\min\{d^{2+\alpha^+},\ d^{2+\alpha^-}\}[D^2u]_{0,\alpha(\cdot),D}\leq C\left(|u|_{0,\Omega}+|f|_{0,\alpha(\cdot),\Omega}\right)
\end{align*}
is satisfied, where $C=C(\textup{diam}(\Omega), n, \alpha^-, \alpha^+, c_{\log}(\alpha), \Lambda, \lambda)$.
\end{lem}
\begin{proof}
Since  the coefficients of $L$ satisfies the assumptions of Theorem \ref{sches}, we obtain
\begin{equation*}
|u|^*_{2,\alpha(\cdot), \Omega}\leq C\left(|u|_{0, \Omega}+|f|^{(2)}_{0,\alpha(\cdot), \Omega}\right).
\end{equation*}
Moreover, by the very definitions of the norms we get 
\begin{align*}
|u|^*_{2,\alpha(\cdot), \Omega}\geq|Du|^*_{0,\Omega}+|D^2u|^*_{0,\Omega}+|D^2u|^*_{0,\alpha(\cdot),\Omega}\geq\\ d|Du|_{0,D}+d^2|D^2u|_{0,D}+\min\{d^{2+\alpha^+},\ d^{2+\alpha^-}\}[D^2u]_{0,\alpha(\cdot),D}.
\end{align*}
This completes the proof of the lemma.
\end{proof}
By the same methods as in the proof of Theorem \ref{sches} we obtain.
\begin{theo}\label{schesbrz}
Let $\Omega\subset\mathbb{R}^n_+$ be an open and bounded set with boundary portion $T$ in $\{x_n=0\}$. 
If $u\in C^{2,\alpha(\cdot)}(\bar{\Omega)}$ satisfies
\begin{equation*}\
Lu=f,\quad u|_{T}=0,
\end{equation*}
where $f\in C^{\alpha(\cdot)}(\bar{\Omega})$ and there are positive constants $\lambda$ and $\Lambda$ such that
\begin{align*}
&a^{ij}(x)\zeta^i\zeta^j\geq\lambda|\zeta|^2\qquad\textrm{for all $x\in\Omega$ and for all $\zeta\in\mathbb{R}^n$,}\\
&|a^{ij}|_{0,\alpha(\cdot),\Omega\cup T}^{(0)},|b^i|_{0,\alpha(\cdot),\Omega\cup T}^{(1)},|c|^{(2)}_{0,\alpha(\cdot),\Omega\cup T}\leq \Lambda,
\end{align*} then
\begin{equation*}
|u|^*_{2,\alpha(\cdot), \Omega\cup T}\leq C\left(|u|_{0,\Omega}+|f|^{(2)}_{0,\alpha(\cdot), \Omega\cup T}\right),
\end{equation*}
where $C=C(\textup{diam}(\Omega), n, \alpha^-, \alpha^+, c_{\log}(\alpha), \Lambda, \lambda)$.
\end{theo}

Now, we can formulate the following lemma.
\begin{lem}\label{schesbrzex}
Let $\Omega\subset\mathbb{R}^n$ be an open and bounded set of class $C^{2,\alpha^+}$. If $u\in C^{2,\alpha(\cdot)}(\bar{\Omega})$ satisfies
\begin{align*}
\left\{\begin{array}{l}
Lu=f \textrm{in $\Omega$,}\\
u=0\textrm{ on $\partial\Omega$,}
\end{array}\right.
\end{align*}
where $f\in C^{\alpha(\cdot)}(\bar{\Omega})$ and there are positive constants $\lambda$ and $\Lambda$ such that
\begin{align*}
&a^{ij}(x)\zeta^i\zeta^j\geq\lambda|\zeta|^2\qquad\textrm{for all $x\in\Omega$ and for all $\zeta\in\mathbb{R}^n$,}\\
&|a^{ij}|_{0,\alpha(\cdot),\Omega},|b^i|_{0,\alpha(\cdot),\Omega},|c|_{0,\alpha(\cdot),\Omega}\leq \Lambda,
\end{align*} 
then there exists $\rho$ such that for each balls $B=B(x_0,\rho)$, where $x_0\in\partial\Omega$, the following estimate
\begin{equation*}
|u|_{2,\alpha(\cdot), \Omega\cap B}\leq C\left(|u|_{0,\Omega}+|f|_{0,\alpha(\cdot), \Omega}\right)
\end{equation*}
holds, where $C=C(\Omega, n, \alpha^-, \alpha^+, c_{\log}(\alpha), \Lambda, \lambda)$.
\end{lem}
\begin{proof}
Let us fix $x_0\in\partial\Omega$. Since $\partial \Omega$ is of class $C^{2,\alpha^+}$, there exist $\rho_{x_0}>0$ and mapping $\Phi\colon B(x_0,\rho_{x_0})\to\mathbb{R}^n$ of class $C^{2,\alpha^+}$ such that $\Phi(\partial\Omega \cap B(x_0,\rho_{x_0}))\subset\{x_n=0\}$ and $\Phi(\Omega\cap B(x_0,\rho_{x_0}))\subset\mathbb{R}^n_+$.
 Let us introduce the notation $B'=B(x_0,\rho_{x_0})\cap\Omega$, $\Phi(B')=D'$, $T=B(x_0,\rho_{x_0})\cap\partial\Omega$ and $T'=\Phi(T)$. 
Let moreover define $\tilde{u}=u\circ\Phi^{-1}$ and $\tilde{f}=f\circ\Phi^{-1}$ on the set $D'$. Then $\tilde{u}$ satisfies the equation
\begin{align*}
\widetilde{L}\tilde{u}=\tilde{a}^{ij}D_{ij}\tilde{u}+\tilde{b}^iD_i\tilde{u}+\tilde{c}\tilde{u}=\tilde{f},
\end{align*}
where
\begin{eqnarray*}
\tilde{a}^{ij}=\left(\sum_{k,l=1}^na^{kl}D_l\Phi^jD_k\Phi^i\right)\circ\Phi^{-1}, \quad 
\tilde{b}^i=\left(\sum_{k,l=1}^nD_{lk}\Phi^ia^{lk}+\sum_{k=1}^nb^kD_k\Phi^i\right)\circ\Phi^{-1},\quad
\tilde{c}=c\circ\Phi^{-1}.
\end{eqnarray*}
It is easy to see that there exists constant $K>0$ that
\begin{align*}
K^{-1}|x-y|\leq |\Phi(x)-\Phi(y)|\leq K|x-y|.
\end{align*}
Consequently, similarly as in the proof o Lemma \ref{lemful1}, we obtain
\begin{equation}\label{pominbrz1}
\begin{split}
C_2|v|^*_{k,{\alpha}(\cdot),B'}\leq|\tilde{v}|^*_{k,\tilde{\alpha}(\cdot),D'}\leq C_1|v|^*_{k,{\alpha}(\cdot),B'}\\
C_2|v|^{(l)}_{k,\alpha(\cdot),B'}\leq|\tilde{v}|^{(l)}_{k,\tilde{\alpha}(\cdot),D'}\leq C_1|v|^{(l)}_{k,\alpha(\cdot),B'},
\end{split}
\end{equation}
and
\begin{equation}\label{pominbrz}
\begin{split}
C_2|v|^*_{k,\alpha(\cdot),B'\cup T}\leq|\tilde{v}|^*_{k,\tilde{\alpha}(\cdot),D'\cup T'}\leq C_1|v|^*_{k,\alpha(\cdot),B'\cup T}\\
C_2|v|^{(l)}_{k,\alpha(\cdot),B'\cup T}\leq|\tilde{v}|^{(l)}_{k,\tilde{\alpha}(\cdot),D'\cup{T}'}\leq C_1|v|^{(l)}_{k,\alpha(\cdot),B'\cup T},
\end{split}
\end{equation}
for $k,\ l=0,\ 1,\ 2,\ 3,\ldots $, where $v$ is a certain function, $\tilde{v}=v\circ\Phi^{-1}$ and $\tilde{\alpha}=\alpha\circ\Phi^{-1}$. 

Hence, we see that 
\begin{align*}
|\tilde{a}^{ij}|_{0,\tilde{\alpha}(\cdot),D'},|\tilde{b}^i|_{0,\tilde{\alpha}(\cdot),D'},|\tilde{c}|_{0,\tilde{\alpha}(\cdot),D'}\leq \widetilde{\Lambda}=C\Lambda.
\end{align*}
Thus, by virtue of Theorem \ref{schesbrz} we get
\begin{equation*}
|\tilde{u}|^*_{2,\tilde{\alpha}(\cdot), D'\cup T'}\leq C\left(|\tilde{u}|_{0,D'}+|\tilde{f}|_{0,\tilde{\alpha}(\cdot),D'\cup T'}^{(2)}\right).
\end{equation*}
Therefore, from (\ref{pominbrz1}) and (\ref{pominbrz}) we obtain
\begin{align}\label{schesexin1}
|u|^*_{2,\alpha(\cdot),B'\cup T}\leq |\tilde{u}|^*_{2,\tilde{\alpha}(\cdot), D'\cup T'}\leq C\left(|\tilde{u}|_{0,D'}+|\tilde{f}|_{0,\tilde{\alpha}(\cdot),D'\cup T'}^{(2)}\right)\leq \\
 C\left(|u|_{0,B'}+|f|^{(2)}_{0,\alpha(\cdot),B' \cup T}\right)\leq  C\left(|u|_{0,B'}+|f|_{0,\alpha(\cdot),B'}\right) \leq C\left(|u|_{0,\Omega}+|f|_{0,\alpha(\cdot),\Omega}\right).\nonumber
\end{align}
Next, let $B''=B\left(x_0,\frac{\rho_{x_0}}{2}\right)\cap\Omega$. We see that $\bar{d}_x, \bar{d}_{x,y}\geq \frac{\rho_{x_0}}{2}$ for all $x,\ y\in B'$ and thus we conclude
\begin{align*}
C(\rho_{x_0})|u|_{2,\alpha(\cdot),B''}\leq|u|^*_{2,\alpha(\cdot), B'\cup T}.
\end{align*}
According to inequality (\ref{schesexin1}), we have
\begin{align}\label{schesexin2}
|u|_{2,\alpha(\cdot),B''} \leq C(\rho_{x_0})\left(|u|_{0,\Omega}+|f|_{0,\alpha(\cdot),\Omega}\right).
\end{align}
Now, let us take the covering  $\{B(x,\frac{\rho_x}{4})\}_{x\in\partial\Omega}$ of the boundary $\partial\Omega$. 
Since $\partial\Omega$ is compact, we can take a finite cover  $\{B(x_i,\frac{\rho_i}{4})\}_{i=1}^{N}$ of $\partial\Omega$. Let $\rho=\min\frac{\rho_i}{4}$, 
if we take an arbitrary $x_0\in\partial\Omega$, then $x_0\in B(x_i, \frac{\rho_i}{4})$ for some $i$. It is easy to see that $B(x_0,\rho)\subset B(x_i, \frac{\rho_i}{2})$ 
and since for $B''=B(x_i,\frac{\rho_i}{2})\cap\Omega$ inequality (\ref{schesexin2}) holds, thus the proof follows.
\end{proof}
Our next main result is the following claim.
\begin{theo}[Global Schauder estimates]\label{schesend}
Let $\Omega\subset\mathbb{R}^n$ be an open and bounded set of class $C^{2,\alpha^+}$. If $u\in C^{2,\alpha(\cdot)}(\bar{\Omega})$ satisfies
\begin{align*}\label{probl}
\left\{\begin{array}{l}
Lu=f \textrm{in $\Omega$,}\\
u=\phi\textrm{ on $\partial\Omega$,}
\end{array}\right.
\end{align*}
where $f\in C^{\alpha(\cdot)}(\bar{\Omega})$, $\phi\in C^{2,\alpha(\cdot)}(\bar{\Omega})$ and there are positive constants $\lambda$ and $\Lambda$ such that
\begin{align*}
&a^{ij}(x)\zeta^i\zeta^j\geq\lambda|\zeta|^2\qquad\textrm{for all $x\in\Omega$ and for all $\zeta\in\mathbb{R}^n$,}\\
&|a^{ij}|_{0,\alpha(\cdot),\Omega},|b^i|_{0,\alpha(\cdot),\Omega},|c|_{0,\alpha(\cdot),\Omega}\leq \Lambda,
\end{align*}
 then the following inequality  
\begin{equation}\label{tezsches}
|u|_{2,\alpha(\cdot), \Omega}\leq C\left(|u|_{0,\Omega}+|f|_{0,\alpha(\cdot), \Omega}+|\phi|_{2,\alpha(\cdot), \Omega}\right)
\end{equation}
is satisfied, where $C=C(\Omega, n, \alpha^-, \alpha^+, c_{\log}(\alpha), \Lambda, \lambda)$.
\end{theo}
\begin{proof}
We can assume that $\phi \equiv 0$. Now, let us take $\rho>0$ from Theorem \ref{schesbrzex} and let $x\in \Omega$. Then, by 
Lemma \ref{schesbrzex} and by Lemma \ref{pomlem} we get 
\begin{align*}
|Du(x)|+|D^2u(x)|\leq C\left(|u|_{0,\Omega}+|f|_{0,\alpha(\cdot), \Omega}\right).
\end{align*}
Next, let us take two points $x, y\in\Omega$. 

If $x,y\in B(x_0,\rho)$ for some $x_0\in\partial\Omega$, then  we apply Lemma \ref{schesbrzex} and if $d_{x,y} \geq\frac{\rho}{2}$ we use Lemma \ref{pomlem} getting the 
following estimate  
\begin{align*}
\frac{|D^2u(x)-D^2u(y)|}{|x-y|^{\alpha(x)}}\leq C\left(|u|_{0,\Omega}+|f|_{0,\alpha(\cdot), \Omega}\right).
\end{align*}

If $d_x < \frac{\rho}{2}$ and for any $x_0\in\partial\Omega$,  $x,y\notin B(x_0,\rho)$, then there exists 
$x_0\in\partial\Omega$ such that $x\in B(x_0, \frac{\rho}{2})$. Hence, we have
\begin{align*}
  |x-y|\geq |y-x_0|-|x-x_0|\geq\rho-\frac{\rho}{2}=\frac{\rho}{2},
\end{align*}
in this way we get
\begin{align*}
 \frac{|D^2u(x)-D^2u(y)|}{|x-y|^{\alpha(x)}}\leq\left(\frac{2}{\rho}\right)^{\alpha(x)}\left(|D^2u(x)|+|D^2u(y)|\right)\leq C\left(|u|_{0,\Omega}+|f|_{0,\alpha(\cdot), \Omega}\right).
\end{align*}
This completes the proof of Theorem \ref{schesend}.
\end{proof}

\section{Existence and uniqueness of solution}
In this section we prove the so-called Kellog's type theorem. Namely, we prove that the Dirichlet boundary value problem has a unique solution 
in the generalized H\"{o}lder spaces.
\begin{theo}[Existence and uniqueness of solution] \label{kellog}
Let $\Omega\subset\mathbb{R}^n$ be an open and bounded set with the boundary of class $C^{2,\alpha^+}$ and 
let us assume that $L$ is a strictly elliptic operator with 
coefficients in $C^{\alpha(\cdot)}(\bar{\Omega})$ and  $c\leq 0$. If $f\in C^{\alpha(\cdot)}(\bar{\Omega})$  and $\phi\in C^{2,\alpha(\cdot)}(\bar{\Omega})$, 
then the problem 
\begin{align}\label{probl}
\left\{\begin{array}{l}
L u=f \quad \text{in} \, \Omega,\\
u=\phi \quad \text{ on} \, \partial\Omega,
\end{array}\right.
\end{align}
 has a unique solution $u\in C^{2,\alpha(\cdot)}(\bar{\Omega})$.
\end{theo}
Let $\Omega_{\sigma}=\{x\colon \hbox{dist}(x,\Omega)<\sigma\}$ be a $\sigma$-neighborhood of the set $\Omega$. We first prove the lemma about 
the existence of extension operator in the variable H\"{o}lder spaces.  
\begin{lem}\label{extlem}
Let $\Omega\subset\mathbb{R}^n$ be an open and bounded set with the boundary of class $C^2$ and let assume that $\alpha\in\mathcal{A}^{log}(\Omega)$. 
Then, there exists $\sigma>0$ such that there exists $\bar{\alpha}\in\mathcal{A}^{log}({\Omega}_{\sigma})$  such that $\bar{\alpha}|_{\Omega}=\alpha$, $\bar{\alpha}^+=\alpha^+$ ,
 $\bar{\alpha}^-=\alpha^-$ and for any $f\in C^{\alpha({\cdot})}(\bar{\Omega})$, there exists 
$\bar{f}\in C^{\bar{\alpha}(\cdot)}(\bar{\Omega}_{\sigma})$ such that $\bar{f}|_{\Omega}=f$. Moreover, there exists a constant $C=C(\Omega, n, \alpha^-, \alpha^+, c_{\log}(\alpha))$ such that the inequality
\begin{align*}
|\bar{f}|_{0,\bar{\alpha}(\cdot),\Omega_{\sigma}}\leq C|f|_{0,\alpha(\cdot),\Omega}
\end{align*}
holds.
\end{lem}

\begin{proof} From the tubular neighborhood theorem (see \cite{tubu}) there exists $\sigma>0$ such that for 
$x\in\left(\partial\Omega\right)_{{\sigma}}=\{z\in\mathbb{R}^n\colon\hbox{dist}(z,\partial\Omega)<\sigma\}$ (see Fig. 1), 
there exist unique $d>0$ and $x_0\in\partial\Omega$ such that $x=x_0\pm dn(x_0)$, where $n(x_0)$ is an exterior unit normal vector. 

\centerline{\epsfxsize=10cm \epsfbox{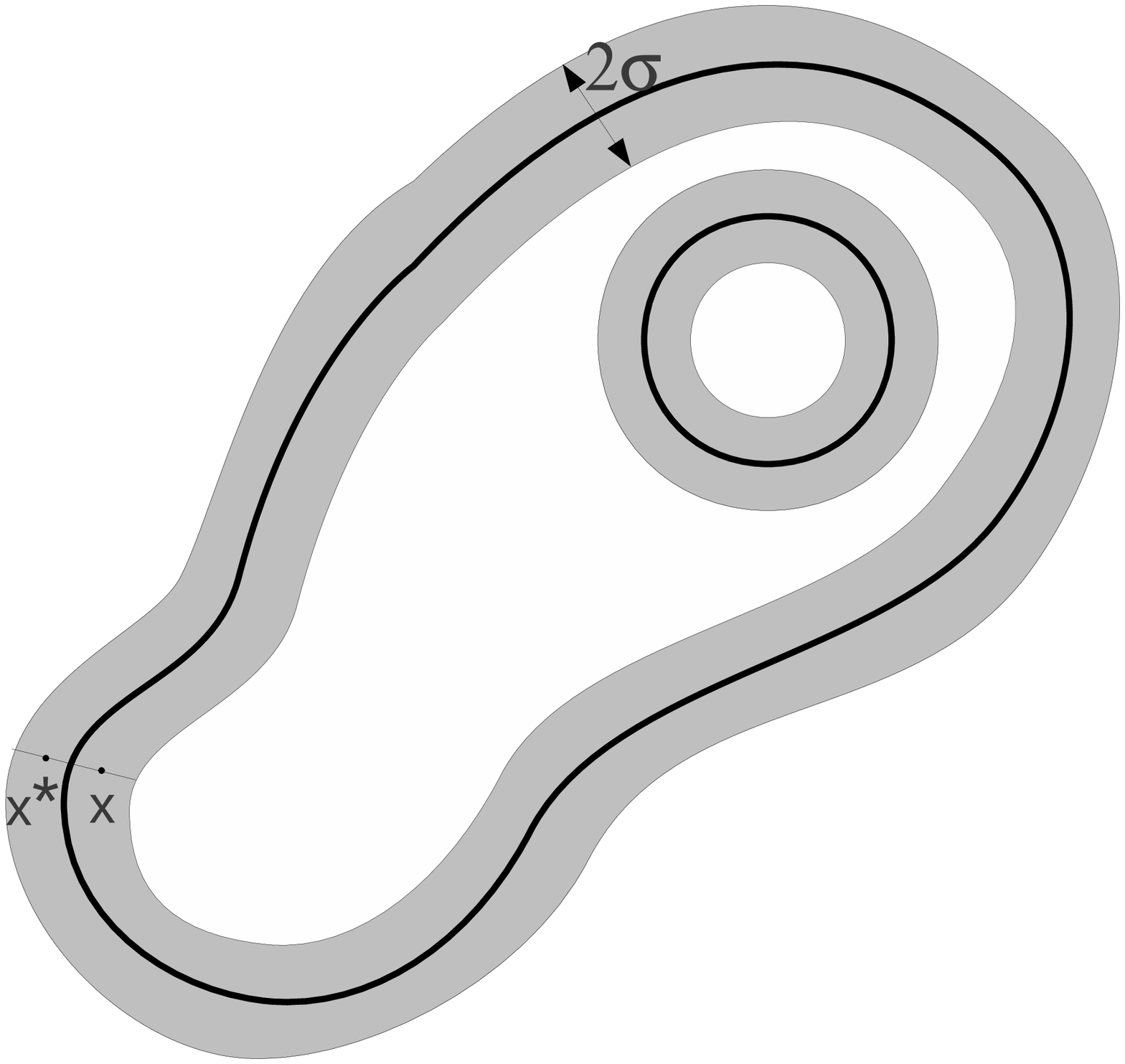}} 
\centerline{Fig. 1}
We define the map
\begin{align*}
*: \left(\partial\Omega\right)_{{\sigma}} \rightarrow \left(\partial\Omega\right)_{{\sigma}}
\end{align*}
 in the following way 
$x^*= x_0+dn(x_0)$, if $x=x_0-dn(x_0)$ and as $x^*=x_0-dn(x_0)$, if $x=x_0+dn(x_0)$. 
Subsequently, we show that the map $*$ is bilipschitz, in other words, there exists constant $K>0$ that 
\begin{align}\label{lipodb}
K^{-1}|x-y|\leq|x^*-y^*|\leq K|x-y|.
\end{align}
Let $x,y\in\left(\partial\Omega\right)_{{\sigma}}$, then there exist $x_0, y_0\in\partial\Omega$ and $d_1, d_2 \in [0, {\sigma}]$ 
such that $x=x_0\pm d_1n(x_0)$ and $y=y_0\pm d_2n(y_0)$. For simplicity we assume that $x=x_0-d_1n(x_0)$ and $y=y_0-d_2n(y_0)$. The case when 
$x=x_0+d_1n(x_0), y=y_0+d_2n(y_0)$ can be treated in the same manner. We see that $x^*=x_0+d_1n(x_0)$ and $y^*=y_0+d_2n(y_0)$. 
Since the boundary of $\Omega$ is smooth, there exists a constant $L>0$ such that $|n(x_0)-n(y_0)|\leq L|x_0-y_0|$ for all $x_0,y_0\in\partial\Omega$. 
From \cite{foote} we also know, that the function $P$ that assign to the every point $z\in(\partial\Omega)_{{\sigma}}$ its projection on $\partial\Omega$ 
is of $C^1$-class. Thus, since $\left(\partial \Omega\right)_{\sigma}$ is relatively compact, the map $P$ is Lipschitz, i.e. $|P(z)-P(v)|\leq M|z-v|$ for all $z,v\in\left(\partial\Omega\right)_{{\sigma}}$. Thus, we get the inequality
\begin{align*}
|x^*-y^*|=|x_0+d_1n(x_0)-y_0-d_2n(y_0)|\leq |x_0-y_0|+d_1|n(x_0)-n(y_0)|+|d_1-d_2|\leq \\ \left(1+\sigma L\right)|x_0+y_0|+|x-y| + |x_0-y_0|\leq
C|x-y|.
\end{align*}
Next, let us assume that $x=x_0+d_1n(x_0), y=y_0-d_2n(y_0)$. We see that $x\notin\Omega$ and $y\in\Omega$. Let $z\in [x,y] \cap \partial\Omega$,
 then we have $|x-x_0|\leq|x-z|$, $|y-y_0|\leq|y-z|$ and $|x-z|+|z-y|=|x-y|$. These give
\begin{align*}
|x^*-y^*|=|x_0-d_1n(x_0)-y_0-d_2n(y_0)|\\ 
\leq|x_0-y_0|+d_1+d_2=|x_0-y_0|+|x-x_0|+|y-y_0|\leq\\ M|x-y|+|x-z|+|y-z|\leq
(1+M)|x-y|.
\end{align*}
In similar way one can show, that $K^{-1}|x-y|\leq|x^*-y^*|$.

Next, let us define $\bar{f}$ and $\bar{\alpha}$ on the set $\Omega_{\sigma}$ as follows
\begin{align*}
\bar{f}(x)=\left\{ 
\begin{array}{cl}
f(x),&\textrm{ if $x\in\bar{\Omega}$}\\
f(x^*),&\textrm{ if $x\in\Omega_{\sigma} \setminus \bar{\Omega}$}
\end{array}\right.&&
\bar{\alpha}(x)=\left\{ 
\begin{array}{cl}
\alpha(x),&\textrm{ if $x\in\bar{\Omega}$}\\
\alpha(x^*),&\textrm{ if $x\in\Omega_{\sigma} \setminus \bar{\Omega}$}
\end{array}\right..
\end{align*}
First, we show that $\bar{\alpha}$ is log-H\"older continuous. Let us assume that $|x-y|\leq \frac{1}{2}$. Then, we have two cases
\begin{enumerate}
\item $x,y\in \Omega_{\sigma}\setminus \Omega$.
Then, we have the following chain of inequalities 
\begin{align*}
\begin{gathered}
\left|\ln\left|x-y\right|\right|\left|\bar{\alpha}(x)-\bar{\alpha}(y)\right|=-\ln\left|x-y\right|\left|\bar{\alpha}(x)-\bar{\alpha}(y)\right|=\\-\ln\left|x-y\right|\left|\alpha(x^*)-{\alpha}(y^*)\right|\leq-\ln \left(K^{-1}\left|x^*-y^*\right|\right)\left|\alpha(x^*)-{\alpha}(y^*)\right|\\
=-\ln \left|x^*-y^*\right|\left|\alpha(x^*)-{\alpha}(y^*)\right|+\ln K\left|\alpha(x^*)-{\alpha}(y^*)\right|\leq\\ \left|\ln \left|x^*-y^*\right|\right|\left|\alpha(x^*)-{\alpha}(y^*)\right|+\left|\ln K\right|2\alpha^+\leq c_{\log}(\alpha)+\left|\ln K\right|2\alpha^+ .
\end{gathered}
\end{align*}
\item $x\in\bar{\Omega}$ and $y\in \Omega_{\sigma}\setminus \bar{\Omega}$.
Let $z\in [x,y] \cap \partial\Omega$, then 
\begin{align*}
\left|\ln\left|x-y\right|\right|\left|\bar{\alpha}(x)-\bar{\alpha}(y)\right| \leq -\ln\left|x-y\right|\left(\left|\bar{\alpha}(x)-\bar{\alpha}(z)\right|+\left|\bar{\alpha}(z)-\bar{\alpha}(y)\right|\right)\\
\leq -\ln\left|x-z\right|\left|\bar{\alpha}(x)-\bar{\alpha}(z)\right|-\ln\left|z-y\right|\left|\bar{\alpha}(z)-\bar{\alpha}(y)\right|,
\end{align*}
and the RHS of the above expression can be easily estimated by the previous case.
\end{enumerate}
When $|x-y|\geq \frac{1}{2}$, then
\begin{align*}
\left|\ln\left|x-y\right|\right|\left|\bar{\alpha}(x)-\bar{\alpha}(y)\right|\leq\max\left\{\left|\ln\frac{1}{2}\right|,\left|\ln\left|\hbox{diam}\left(\Omega\right)_{\sigma}\right|\right|\right\}2\alpha^+.
\end{align*}
In this way we have showed that $\bar{\alpha}\in\mathcal{A}^{log}({\Omega}_{\sigma})$.

Finally, we show that $\bar{f}\in C^{\bar{\alpha}(\cdot)}(\bar{\Omega}_{\sigma})$.
It is easy to see that $\bar{f}$ is continuous. 
Let us assume that $|x-y|\leq \frac{1}{2}$. 
\begin{enumerate}
\item If $x,y\in\Omega_{\sigma}\setminus \Omega$, then thanks to inequality (\ref{lipodb}), we get
\begin{align*}
\frac{|\bar{f}(x)-\bar{f}(y)|}{|x-y|^{\bar{\alpha}(x)}}\leq K^{\alpha(x^*)}\frac{|\bar{f}(x)-\bar{f}(y)|}{|x^*-y^*|^{\bar{\alpha}(x)}}\leq C\frac{|{f}(x^*)-{f}(y^*)|}{|x^*-y^*|^{{\alpha}(x^*)}}.
\end{align*}
\item If $x\in\bar{\Omega}$ and $y\in \Omega_{\sigma}\setminus \bar{\Omega}$, we take $z \in [x,y] \cap \partial\Omega$, then
\begin{align*}
\frac{|\bar{f}(x)-\bar{f}(y)|}{|x-y|^{\bar{\alpha}(x)}}\leq 
\frac{|\bar{f}(x)-\bar{f}(z)|}{|x-y|^{\bar{\alpha}(x)}}+\frac{|\bar{f}(z)-\bar{f}(y)|}{|x-y|^{\bar{\alpha}(x)}}\leq \\
 \frac{|\bar{f}(x)-\bar{f}(z)|}{|x-z|^{\bar{\alpha}(x)}}+\frac{|\bar{f}(z)-\bar{f}(y)|}{|z-y|^{\bar{\alpha}(y)}}\frac{|z-y|^{\bar{\alpha}(y)}}{|x-y|^{\bar{\alpha}(x)}}\leq 
\frac{|\bar{f}(x)-\bar{f}(z)|}{|x-z|^{\bar{\alpha}(x)}}+\frac{|\bar{f}(z)-\bar{f}(y)|}{|z-y|^{\bar{\alpha}(y)}}{|x-y|^{\bar{\alpha}(y)-\bar{\alpha}(x)}},
\end{align*}
where the term $|x-y|^{\bar{\alpha}(y)-\bar{\alpha}(x)}$ is bounded since $\bar{\alpha}\in\mathcal{A}^{log}({\Omega}_{\sigma})$.
\end{enumerate}
If $|x-y|>\frac{1}{2}$, then
\begin{align*}
\frac{|\bar{f}(x)-\bar{f}(y)|}{|x-y|^{\bar{\alpha}(x)}}\leq 2 |\bar{f}|_{0,\Omega_{\sigma}}\leq2|f|_{0,\Omega}.
\end{align*}
According to above inequalities we have
\begin{align*}
 |\bar{f}|_{0,\bar{\alpha}(\cdot),\Omega_{\sigma}}\leq C|f|_{0,\alpha(\cdot),\Omega}.
\end{align*}
This completes the proof of Lemma \ref{extlem}.
\end{proof}

Let $\phi_{\epsilon}$ be the standard mollifier, i.e. $\phi_{\epsilon} \geq 0$, $\hbox{supp} \phi_{\epsilon} \subset B(0,\epsilon)$ and $\int_{\mathbb{R}^n} \phi_{\epsilon} (x) dx =1$. Then, if $f \in L^1_{loc}$, we define 
$f_{\epsilon}= \phi_{\epsilon} \star f$. Now, we are in position to formulate the following approximation lemma.

\begin{lem}\label{wygl}
Assume that $\Omega\subset\mathbb{R}^n$ is an open and bounded set and $\sigma>0$. Let us also 
assume, that $\alpha$ is log-H\"older continuous on $\Omega_{\sigma}$ and that $f\in C^{\alpha(\cdot)}(\bar{\Omega}_{\sigma})$, then 
for each $\delta \in (0, \sigma)$ there exists $\epsilon (\delta) >0$ such that for each $\epsilon \in (0, \epsilon (\delta))$, the inequality
\begin{equation*}
|f_{\epsilon}|_{0,\alpha(\cdot)-\delta,\Omega}\leq 3|f|_{0,\alpha(\cdot),\Omega_{\sigma}}
\end{equation*}
holds.
\end{lem}
\begin{proof}
Let us assume that $x,y\in\Omega$ and $|x-y|\leq 1$, then 
\begin{align}\label{ilin1}
\frac{|f_{\epsilon}(x)-f_{\epsilon}(y)|}{|x-y|^{\alpha(x)}}\leq\int_{B(0,\epsilon)}\phi_{\epsilon}(z)\frac{|f(x-z)-f(y-z)|}{|x-y|^{\alpha(x)}}dz=\\
\int_{B(0,\epsilon)}\phi_{\epsilon}(z)\frac{|f(x-z)-f(y-z)|}{|(x-z)-(y-z)|^{\alpha(x-z)}}|x-y|^{\alpha(x-z)-\alpha(x)}dz. \nonumber
\end{align}
Since $\alpha$ is uniformly continuous, there exists $\epsilon>0$ such that
\begin{align*}
|\alpha(x)-\alpha(y)|<\delta\textrm{, if $|x-y|<\epsilon$}.
\end{align*}
For $z\in B(0,\epsilon)$ we have $\alpha(x-z)-\alpha(x) >-\delta$, thus since $|x-y|<1$, we conclude
\begin{equation*}
|x-y|^{\alpha(x-z)-\alpha(x)}\leq |x-y|^{-\delta}.
\end{equation*}
According to (\ref{ilin1}), it follows that
\begin{align*}
\frac{|f_{\epsilon}(x)-f_{\epsilon}(y)|}{|x-y|^{\alpha(x)}}\leq \int_{B(0,\epsilon)}\phi_{\epsilon}(z)\frac{|f(x-z)-f(y-z)|}{|(x-z)-(y-z)|^{\alpha(x-z)}}|x-y|^{\alpha(x-z)-\alpha(x)}dz\leq\\
|x-y|^{-\delta}[f]_{0,\alpha(\cdot),\Omega_{\sigma}}.
\end{align*}
Thus, we get
\begin{align*}
\frac{|f_{\epsilon}(x)-f_{\epsilon}(y)|}{|x-y|^{\alpha(x)-\delta}}\leq[f]_{0,\alpha(\cdot),\Omega_{\sigma}}.
\end{align*}
Moreover, it easy to see that
\begin{equation*}
|f_{\epsilon}|_{0,\Omega}\leq|f|_{0,\Omega_{\sigma}},
\end{equation*}
hence for $|x-y| \geq 1$ we obtain
\begin{equation*}
\frac{|f_{\epsilon}(x)-f_{\epsilon}(y)|}{|x-y|^{\alpha(x)-\delta}}\leq 2|f_{\epsilon}|_{0,\Omega}\leq 2 |f|_{0,\Omega_{\sigma}}.
\end{equation*}
Therefore, we get
\begin{equation*}
[f_{\epsilon}]_{0,\alpha(\cdot)-\delta,\Omega}\leq([f]_{0,\alpha(\cdot)\Omega_{\sigma}}+2|f|_{0,\Omega_{\sigma}}).
\end{equation*}
Finally, we conclude
\begin{equation*}
|f_{\epsilon}|_{0,\alpha(\cdot)-\delta,\Omega}\leq 3|f|_{0,\alpha(\cdot),\Omega_{\sigma}},
\end{equation*}
which completes the proof of Lemma \ref{wygl}.
\end{proof}

\begin{proof}[Proof of Theorem \ref{kellog}]
Since the equation is linear, we can assume that the boundary values are equal to zero.

First of all we consider a very special case of the operator $L$, namely we shall take $L=\Delta$, i.e. we consider the Poisson equation
\begin{align*}
\left\{\begin{array}{l}
\Delta u=f \quad \text{in} \,\Omega,\\
u=0 \quad\text{on} \,\partial\Omega.
\end{array}\right.
\end{align*}
By Lemma \ref{extlem}, we can extend $f$ and $\alpha$ on the set $\Omega_{\sigma}$ for some $\sigma>0$. We shall denote the extensions by $\bar{f}$ and $\bar{\alpha}$ respectively.
Next, by Lemma \ref{wygl} there exist sequences $\epsilon_m$ and $\delta_m$ such that both are monotone and convergent to $0$ and the inequality
\begin{align}\label{inex1}
|\bar{f}_{\epsilon_m}|_{0,\bar{\alpha}(\cdot)-\delta_m,\Omega}\leq 3|\bar{f}|_{0,\bar{\alpha}(\cdot),\Omega_{\sigma}}
\end{align}
is satisfied for all $m$. Functions $\bar{f}_{\epsilon_m}$ are smooth, thus in particular $\bar{f}_{\epsilon_m}\in C^{\alpha^+-\delta_m}(\bar{\Omega})$. 
Therefore, by the classical Schauder theory there exists $u_{\epsilon_m}\in C^{2,\alpha^+-\delta_m}(\bar{\Omega})$ 
such that $\Delta u_{\epsilon_m}=f_{\epsilon_m}$. Furthermore, by Theorem \ref{schesend} the following inequality holds
\begin{align*}
|u_{\epsilon_m}|_{2,\alpha(\cdot)-\delta_m,\Omega}\leq C|\bar{f}_{\epsilon_m}|_{0,\alpha(\cdot)-\delta_m,\Omega}.
\end{align*}
Consequently, gathering the above inequality with (\ref{inex1}) we have
\begin{align}\label{inex2}
|u_{\epsilon_m}|_{2,\alpha(\cdot)-\delta_m,\Omega}\leq C|\bar{f}_{\epsilon_m}|_{0,\alpha(\cdot)-\delta_m,\Omega} \leq C|\bar{f}|_{0,\bar{\alpha}(\cdot),\Omega_{\sigma}}\leq C|f|_{0,\alpha(\cdot),\Omega},
\end{align}
where the last inequality follows from the property of the extension operator $\bar{f}$. Since $\alpha^- >0$, we can assume 
that there exists $\gamma>0$ such that for all $m$ we have $\alpha^--\delta_m\geq \gamma$. Then, we have 
\begin{align*}
|u_{\epsilon_m}|_{2,\gamma,\Omega}\leq|u_{\epsilon_m}|_{2,\alpha(\cdot)-\delta_m,\Omega}\leq C|f|_{0,\alpha(\cdot),\Omega}.
\end{align*}
It yields that $u_{\epsilon_m}$ is bounded in the space $C^{2,\gamma}(\bar{\Omega})$. So, by virtue of the Arzela-Ascoli Theorem, 
we can conclude that there exists a subsequence, still denoted as $u_{\epsilon_m}$, and $u\in C^2(\Omega)$ such that
\begin{align*}
u_{\epsilon_m}\to u\textrm{ in }C^2(\Omega). 
\end{align*}
Letting $m\to\infty$ in equality $\Delta u_{\epsilon_m}=f_{\epsilon_m}$, we obtain
\begin{align*}
\Delta u=f.
\end{align*}
Moreover, by virtue of inequality (\ref{inex2}) there exists a constant $M$ such that for arbitrary $x,y\in\Omega$, $x\neq y$
\begin{align*}
\frac{|D^2u_{\epsilon_m}(x)-D^2u_{\epsilon_m}(y)|}{|x-y|^{\alpha(x)-\delta_m}}\leq M.
\end{align*}
Let now $m\to\infty$ in the above inequality, then we conclude $u\in C^{2,\alpha(\cdot)}(\bar{\Omega})$. This finishes the proof for $L= \Delta$. 

In order to consider the general case we apply the method of continuity. Let $L_1=L$ and $L_0=\Delta$ and $L_t=(1-t)L_0+tL_1$ for $t\in (0,1)$. By Theorem $\ref{schesend}$ we have the inequality
\begin{align}\label{crul1}
|u|_{2,\alpha(\cdot),\Omega}\leq C\left(|u|_{0,\Omega}+|L_tu|_{0,\alpha(\cdot),\Omega}\right),
\end{align}
for arbitrary $t\in [0,1]$ and $u\in C^{2,\alpha(\cdot)}(\Omega)$. Since $c \leq 0$, by the maximum principle we have that $|u|_{0,\Omega}\leq C|L_tu|_{0,\Omega}$. 
Thus, from (\ref{crul1}) we obtain that
\begin{align*}
|u|_{2,\alpha(\cdot),\Omega}\leq C|L_tu|_{0,\alpha(\cdot),\Omega}. 
\end{align*}
Finally, taking the appropriate Banach spaces in the method of continuity we get the existence and uniqueness of solution. This completes the proof of the theorem. 
\end{proof}
Let us close our discussion with the following example.
\begin{example}
Let $e^{-2}<\gamma<\zeta<1$ and $\Omega=\{x\in\mathbb{R}^n\colon \gamma<|x|<\zeta\}$. Moreover, let $ \alpha: \Omega \rightarrow (0,1]$ be a variable exponent defined as follows $\alpha(x)=|x|$. 
Define 
\begin{align*}
f(x)=\left(|x|-\gamma\right)^{|x|}.
\end{align*}
Then, the variable exponent is log-H\"{o}lder continuous. Indeed, we have 
\begin{align*}
|\alpha(x)-\alpha(y)||\ln|x-y||=||x|-|y|||\ln|x-y||\leq|x-y||\ln|x-y||\leq \sup_{\gamma<r<\zeta} |r\ln r|<\infty.
\end{align*}
One can check by direct computation that for $x,y \in \bar{\Omega}$, the following inequality holds
\begin{align*}
|f(x)-f(y)|\leq|x-y|^{|x|}.
\end{align*}
Thus, we obtain $f\in C^{\alpha(\cdot)}(\bar{\Omega})$ and by virtue of Theorem \ref{kellog} we get that the problem 
\begin{align}\label{przyklad}
\left\{\begin{array}{l}
\Delta u=f \quad \text{in} \, \Omega,\\
u=0 \quad \text{ on} \, \partial\Omega,
\end{array}\right.
\end{align}
has a unique solution $u\in C^{2,\alpha(\cdot)}(\bar{\Omega})$.

On the other hand, we have that $f\notin C^{\beta}(\bar{\Omega})$ for any $\beta\in (\gamma,\zeta]$. Indeed, let us fix $\beta \in (\gamma, \zeta]$ and $x_0=(\gamma, 0,..., 0)$, $x_n=(\gamma + \frac{\beta-\gamma}{2^n}, 0, ..., 0)$. 
Therefore, we get 
\begin{align*}
\frac{|f(x_0)-f(x_n)|}{|x_0-x_n|^{\beta}}\geq \left(\frac{\beta-\gamma}{2^n}\right)^{\frac{\gamma-\beta}{2}}\to \infty.
\end{align*}
Hence, for any $\beta \in (\gamma, \zeta]$ there are no solutions of equation (\ref{przyklad}) in the space $C^{2,\beta}(\bar{\Omega})$. 
\end{example}

\section{Appendix}
In this last section we present the interpolation theorems in the variable H\"older spaces. 
\begin{lem}\label{intlem}
Let $j, k=0, 1, 2, \ldots$ and let us suppose that $j+\beta^+<k+\alpha^-$. If $\alpha\equiv0$ or $0<\alpha^-\leq\alpha^+< 1$ 
and if $\beta\equiv 0 $ or $0<\beta^-\leq\beta^+<1$, then for  
$\epsilon>0$ there exists a constant $C=C(\epsilon, \alpha^-, \beta^+, j, k, n)$ such that for each $u\in C^{k,\alpha(\cdot)}(\Omega)$ the inequalities
\begin{eqnarray}
[u]^*_{j,\beta(\cdot),\Omega}\leq C|u|_{0,\Omega}+\epsilon[u]^*_{k,\alpha(\cdot),\Omega},\label{inttez1}\\ \label{inttez2} 
|u|^*_{j,\beta(\cdot),\Omega}\leq C|u|_{0,\Omega}+\epsilon[u]^*_{k,\alpha(\cdot),\Omega},
\end{eqnarray}
are satisfied.
\end{lem}
\begin{proof}
Inequalities (\ref{inttez1}) and (\ref{inttez2}) are obvious, when $j=\beta=0$. Since inequality (\ref{inttez2}) follows from 
inequality (\ref{inttez1}), we give the proof of (\ref{inttez1}). The proof of (\ref{inttez1}) for $\alpha=\beta=0$ and $j=1$ and $k=2$ 
\begin{equation}\label{inbf}
[u]^*_{1,\Omega} \leq C|u|_{0,\Omega}+\epsilon[u]^*_{2,\Omega}
\end{equation}
is given in \cite{trudinger}. We start with the case $j=k=2$, $\alpha^->0$ and $\beta=0$. 
Let us take $x\in\Omega$ and $\mu\leq\frac{1}{2}$ and define $d=\mu d_x$, $B=B(x,d)$. In addition, let $[x_1, x_2]$ be the interval of the 
length $2d$, parallel to axis $x_j$ such that $x= \frac{x_1 +x_2}{2}$. Then, by the Mean Value Theorem, there exists $z \in [x_1, x_2]$ such that
$$
|D_{ij}u(z)|=|\frac{D_i(x_1)-D_i(x_2)}{2d}|\leq\frac{1}{d}|D_iu|_{0,B}.
$$
Thus, we obtain
\begin{eqnarray}\label{intinp1}
\begin{gathered}
|D_{ij}u(x)|\leq|D_{ij}u(x)-D_{ij}u(z)|+|D_{ij}u(z)|\leq\\ d^{\alpha(x)}\sup_{y\in B}\frac{|D_{ij}u(x)-D_{ij}u(y)|}{|x-y|^{\alpha(x)}}+\frac{1}{d}|D_iu|_{0,B}\leq \\
d^{\alpha(x)}\sup_{y\in B}d_{x,y}^{-\alpha(x)-2}\sup_{y\in B}d_{x,y}^{\alpha(x)+2}\frac{|D_{ij}u(x)-D_{ij}u(y)|}{|x-y|^{\alpha(x)}}+\frac{1}{d}\sup_{y\in B}d_y^{-1}\sup_{y\in B}|d_yD_iu(y)|.
\end{gathered}
\end{eqnarray}
Since $d_{x,y},d_y> d_x-d=d_x-\mu d_x\geq\frac{1}{2}d_x$, we have
\begin{eqnarray*}
|D_{ij}u(x)|\leq \frac{\mu^{\alpha(x)}2^{\alpha(x)+2}}{d_x^2}\sup_{y\in B}d_{xy}^{\alpha(x)+2}\frac{|D_{ij}u(x)-D_{ij}u(y)|}{|x-y|^{\alpha(x)}}+\frac{2}{d_x^2\mu}\sup_{y\in B}|d_yD_iu(y)|,
\end{eqnarray*}
and finally it yields that 
\begin{equation}\label{intin1}
d_x^2|D_{ij}u(x)|\leq 8\mu^{\alpha^-}[u]^*_{2,\alpha(\cdot), \Omega}+\frac{2}{\mu}[u]_{1,\Omega}^*.
\end{equation}
Since $x\in\Omega$ and $\mu\leq\frac{1}{2}$ are arbitrary, we obtain the following inequality
\begin{equation}\label{intin2}
[u]_{2,\Omega}^*\leq \tilde{C}[u]_{1,\Omega}^*+\epsilon[u]^*_{2,\alpha(\cdot), \Omega}.
\end{equation}
Moreover, by the very similar consideration, the following inequality follows
\begin{equation*}
[u]_{1,\Omega}^*\leq C|u|_{0,\Omega}^*+\epsilon[u]^*_{1,\alpha(\cdot), \Omega}.
\end{equation*}
It is exactly inequality (\ref{inttez1}) with $j=k=1$ and $\beta=0$. Now, let us plug $\epsilon=\frac{1}{2\tilde{C}}$ in inequality (\ref{inbf}) and thus from 
(\ref{intin2}) we get 
$$
[u]_{2,\Omega}^*\leq \tilde{C}\left(D|u|_{0,\Omega}+\frac{1}{2\tilde{C}}[u]_{2,\Omega}^*\right)+\epsilon[u]^*_{2,\alpha(\cdot), \Omega}.
$$
It yields the inequality in the case $j=k=2$ and $\beta=0$. Subsequently, gathering (\ref{inbf}) with the above inequality we get inequality (\ref{inttez1}) 
when $j=1$, $k=2$ and $\beta=0$. 
Next, we prove (\ref{inttez1}) when $j=0$, $k=1$ and $\alpha=0$. Let us take $x,y\in\Omega$ such that $d_x\leq d_y$ and let $d,\mu$ and $B$ be such as before.
Let us assume that $y\in B$, then
\begin{eqnarray*}
d_{x,y}^{\beta(x)}\frac{|u(x)-u(y)|}{|x-y|^{\beta(x)}}\leq d_x^{\beta(x)}|x-y|^{1-\beta(x)}|Du|_{0,B}\leq\mu^{1-\beta(x)}d_x|Du|_{0,B}\\ \leq \mu^{1-\beta(x)}d_x\sup_{z\in B}d_z^{-1}[u]^*_{1,B}\leq 2\mu^{1-\beta(x)}[u]^*_{1,B}.
\end{eqnarray*}
If $y\notin B$, then
\begin{eqnarray}\label{prz}
d_x^{\beta(x)}\frac{|u(x)-u(y)|}{|x-y|^{\beta(x)}}\leq 2\mu^{-\beta(x)}|u|_{0,\Omega}.
\end{eqnarray}
Consequently, we have
\begin{eqnarray}\label{inttez4}
[u]_{0,\beta(\cdot)}\leq 2\mu^{-\beta^+}|u|_{0,\Omega}+2\mu^{1-\beta^+}[u]^*_1.
\end{eqnarray}
It finishes the proof of (\ref{inttez1}) in the case $j=0$, $k=1$ and $\alpha=0$. Next, combining (\ref{inbf}) with (\ref{inttez4}) 
we get our inequality for $j=0$ $k=2$, $\alpha=0$. Moreover, arguing in similar fashion we deduce (\ref{inttez1}) for $j=1$, $k=2$ and $\alpha=0$.

Finally, we prove the inequality in the case when $j=k=0$ and $\beta^+<\alpha^-$.
If $y\in B$,
\begin{eqnarray*}
d_x^{\beta(x)}\frac{|u(x)-u(y)|}{|x-y|^{\beta(x)}}=d_x^{\beta(x)}|x-y|^{\alpha(x)-\beta(x)}\frac{|u(x)-u(y)|}{|x-y|^{\alpha(x)}}\leq \mu^{\alpha(x)-\beta(x)}d_x^{\alpha(x)}\frac{|u(x)-u(y)|}{|x-y|^{\alpha(x)}}.
\end{eqnarray*}
Next, if $y\notin B$
\begin{eqnarray*}
d_x^{\beta(x)}\frac{|u(x)-u(y)|}{|x-y|^{\beta(x)}}\leq 2\mu^{-\beta(x)}|u|_{0,\Omega}.
\end{eqnarray*}
Gathering the above inequalities, it may be concluded that 
\begin{eqnarray*}
[u]^*_{0,\beta(\cdot),\Omega}\leq 2\mu^{-\beta^+}|u|_{0,\Omega}+\mu^{\alpha^--\beta^+}[u]_{0, \alpha(\cdot), \Omega}.
\end{eqnarray*}
Thus, we proved the lemma in the cases when $0\leq j\leq k\leq 2$. The general case can be proved by induction.
\end{proof}
 
\begin{lem}\label{intlembrz}
Let $\Omega$ be an open subset of $\mathbb{R}^n_+$ with boundary portion T on $x_n=0$. 
Assume that $j, k=0, 1, 2,\ \ldots$ and let us suppose that $j+\beta^+<k+\alpha^-$. 
If $\alpha\equiv0$ or $0<\alpha^-\leq\alpha^+< 1$ 
and if $\beta\equiv 0 $ or $0<\beta^-\leq\beta^+<1$, then for  
$\epsilon>0$, there exists a constant $C=C(\epsilon, \alpha^-, \beta^+, j, k, n)$ such that for each  $u\in C^{k,\alpha(\cdot)}(\Omega\cup T)$
the following estimates hold
\begin{eqnarray}
[u]^*_{j,\beta(\cdot),\Omega\cup T}\leq C|u|_{0,\Omega}+\epsilon[u]^*_{k,\alpha(\cdot),\Omega\cup T},\label{inttezbrz1}\\ \label{inttezbrz2} 
|u|^*_{j,\beta(\cdot),\Omega\cup T}\leq C|u|_{0,\Omega}+\epsilon[u]^*_{k,\alpha(\cdot),\Omega\cup T}.
\end{eqnarray}
\end{lem}
\begin{proof}
Since the proof is very similar to the proof of the previous lemma, we left it to the reader.
\end{proof}
\subsection*{Acknowledgments}
 We wish to thank Ma{\l}gosia for drawing the picture. Authors also want to thank Katarzyna Bies and Pawe\l \ W\'ojcicki for reading preliminary version of this manuscript.

\end{document}